\documentclass[11pt,a4paper,reqno]{amsart}
\usepackage[english]{babel}
\usepackage[utf8]{inputenc} 
\usepackage{fancyhdr}
\usepackage{indentfirst}
\usepackage{color}     
\usepackage{graphicx}   
\usepackage{newlfont}

\usepackage{commath}
\usepackage{amsfonts}
\usepackage{amsmath}
\usepackage{amssymb}
\usepackage{amsthm}
\usepackage{latexsym}
\usepackage{mathrsfs}
\usepackage{mathtools}
\mathtoolsset{showonlyrefs=true}
\usepackage{verbatim}
\usepackage[all]{xy}
\usepackage{lineno}
\usepackage{leftidx}
\usepackage[hidelinks]{hyperref}
\usepackage{units}
\usepackage{xfrac}
\usepackage{faktor}
\usepackage{float}
\usepackage{tikz}
\usetikzlibrary{arrows, positioning, automata}
\usepackage{upgreek}

\usepackage{chngcntr}
\usepackage{apptools}

\usepackage[flushleft]{threeparttable}

\theoremstyle{plain} 
\newtheorem*{theo*}{Theorem}
\newtheorem*{cor*}{Corollary}
\newtheorem*{con*}{Conjecture}
\newtheorem{theo}{Theorem}[section] 
\newtheorem{prop}[theo]{Proposition}
\newtheorem{cor}[theo]{Corollary}
\newtheorem{lem}[theo]{Lemma}
\newtheorem{con}[theo]{Conjecture}

\theoremstyle{definition}

\theoremstyle{definition}
\newtheorem{defin}[theo]{Definition}

\theoremstyle{remark}
\newtheorem{rem}[theo]{Remark}
\newtheorem*{rem*}{Remark}
\newtheorem*{cap*}{Caption}

\AtAppendix{\counterwithin{theo}{section}}

\newcommand{\Z}{\mathbb{Z}}
\newcommand{\Zplus}{\Z_{>0}}
\newcommand{\Zpluseq}{\Z_{\geq0}}
\newcommand{\Q}{\mathbb{Q}}
\newcommand{\R}{\mathbb{R}}
\newcommand{\C}{\mathbb{C}}
\newcommand{\T}{\mathbb{T}}

\newcommand{\Diff}{\operatorname{Diff}}
\newcommand{\Vir}{\mathfrak{Vir}}

\newcommand{\Aut}{\operatorname{Aut}}

\newcommand{\Hom}{\operatorname{Hom}}
\newcommand{\Mod}{\operatorname{Mod}}

\newcommand{\Rep}{\operatorname{Rep}}
\newcommand{\Repu}{\operatorname{Rep}^{\operatorname{u}}}
\newcommand{\Repf}{\operatorname{Rep}^{\operatorname{f}}}

\newcommand{\Com}{\operatorname{Com}}

\newcommand{\J}{\mathcal{J}}
\newcommand{\A}{\mathcal{A}}
\newcommand{\B}{\mathcal{B}}

\newcommand{\F}{\mathcal{F}}

\newcommand{\parzero}{{\overline{0}}}
\newcommand{\parone}{{\overline{1}}}
\newcommand{\half}{\frac{1}{2}}
\newcommand{\Y}{\mathcal{Y}}

\newcommand{\scalar}{(\cdot|\cdot)}
\newcommand{\curlyscalar}{\{\cdot|\cdot\}}

\newcommand{\pairing}{\langle\cdot,\cdot\rangle}

\newcommand{\one}{\mathbf{1}}
\newcommand{\id}{\operatorname{id}}

\begin{document}

\author[Tiziano Gaudio]{Tiziano Gaudio}
	\address{School of Mathematical Sciences, Lancaster University, Lancaster LA1 4YF, UK \\
			E-mail: {\tt t.gaudio2@lancaster.ac.uk}
	}

\title[Unitarity and strong graded locality of holomorphic VOSAs with $c\leq 24$]{Unitarity and strong graded locality of holomorphic vertex operator superalgebras with central charge at most 24
}

\begin{abstract}
	We prove that all nice holomorphic vertex operator superalgebras (VOSAs) with central charge at most 24 and with non-trivial odd part are unitary, apart from the hypothetical ones arising as fake copies of the shorter moonshine VOSA or of the latter tensorized with a real free fermion VOSA.
	Furthermore, excluding the ones with central charge 24 of glueing type III and with no real free fermion, we show that they are all strongly graded-local. 
	In particular, they naturally give rise to holomorphic graded-local conformal nets. 
	In total, we are able to prove that 910 of the 969 nice holomorphic VOSAs with central charge 24 and with non-trivial odd part are strongly graded-local, without counting hypothetical fake copies of the shorter moonshine VOSA tensorized with a real free fermion VOSA.
\end{abstract}

\maketitle
	
\tableofcontents

\section{Introduction}

One of the motivations for working with the theory of vertex operator superalgebras (VOSAs) is its role as \textit{purely algebraic} axiomatization of chiral Conformal Field Theory (CFT).
Indeed, a vertex operator can be interpreted as one of the two chiral components, living on a light-ray, of a conformally covariant two-dimensional quantum field. 
From a more mathematical point of view, vertex operators are algebraic data deriving from a chiral CFT version of a Wightman field theory \cite{SW64}, see e.g.\ \cite[Chapter 1]{Kac01} and \cite{RTT22, CRTT}.
Another mathematical axiomatizations of chiral CFT is the so called \textit{operator algebraic approach} via the Haag-Kastler axioms \cite{Haa96, Ara10}, producing graded-local conformal nets of von Neumann algebras, see e.g.\ \cite{GF93, CKL08} and references therein. 
In \cite{CKLW18}, the authors show under which conditions a vertex operator algebra (VOA) gives rise to a (local) conformal net and vice versa. This covers the chiral CFTs of Bose type only, as anticommutation relations are not taken into account. In \cite{CGH}, this correspondence is generalized to involve VOSAs and graded-local conformal nets, so that also Fermi models can be included.
In particular, we now know that it is possible to define in a natural way a graded-local conformal net from a VOSA if the latter satisfies two conditions: the \textit{unitarity}, which is the first step towards the \textit{strong graded locality} (the term ``graded'' is omitted in the Bose cases for obvious reasons).
We stress that almost all of known examples of graded-local conformal nets are obtained from unitary VOSAs. Even more, while not all VOSAs are unitary, there are no examples of unitary VOSAs which are known to not satisfy the strong graded locality. 
On the one hand, von Neumann algebras act on Hilbert spaces and thus we can not expect to construct graded-local conformal nets from non-unitary VOSAs.
On the other hand, it is conjectured that all unitary VOSAs are strongly graded-local.
In support of this conjecture, many examples of unitary VOSAs have been proved to be strongly graded-local, see \cite[Section 8]{CKLW18}, \cite[Section 7]{CGH}, \cite{CTW23, CTW22} and \cite[Section 5]{CGGH23}.

Many interesting VOSA examples are strongly rational. This requirement asks for some technical properties on the VOSA, that is it must be self-contragredient, of CFT type and regular (those properties imply also simplicity, $C_2$-cofiniteness and rationality). 
In particular, strongly rational VOSAs have a finite number of irreducible modules, giving rise, in the local case, to a modular tensor category \cite{HL95I, HL95II, HL95III, Hua95, Hua08}.
A strongly rational VOA is said to be completely unitary if it is strongly unitary, that is it is unitary together with all its modules, and these unitary modules form a \textit{unitary} modular tensor category with a \textit{natural} tensor structure determined in \cite{Gui19I, Gui19II}. A list of completely unitary VOAs with references is given in e.g.\ \cite[Table 1]{CGGH23}.
Moreover, VOSA extensions and their unitarity are well-described in this categorical setting \cite{HKL15, Gui22, CGGH23}, cf.\ also \cite{CKL20, CKM24}. 

The above framework is used in \cite{CGGH23} to prove the unitarity and the strong locality of strongly rational holomorphic VOAs with central charge $c=24$ (a different proof of the unitarity of these VOAs is given in \cite{Lam23}).
Recall that a nice (this is a technical condition whose definition we can ignore at this stage) or just strongly rational VOSA is said to be \textit{holomorphic} if its unique irreducible module is its adjoint module.
In particular, holomorphic VOAs are characterized by the fact that their central charge $c$ must be a positive multiple of $8$, see \cite{DM04b, Zhu96, Sch93}. In the case $c=8$, the only model is the lattice VOA deriving from the lattice $E_8$ of rank $8$; whereas in the case $c=16$, there are the two lattice VOAs coming from the two unimodular even positive-definite lattices $D_{16}$ and $E_8\times E_8$ of rank $16$, see \cite{DM04b}. 
The case $c=24$ is far more difficult and still partially open. It started with the work of Schellekens \cite{Sch93}, stating that if any ``meromorphic $c=24$ CFT'' admits spin-one currents, then they generate a Kac-Moody algebra in a list of 70. 
In other words, it was given the task of proving the existence of exactly 71 holomorphic $c=24$ chiral CFTs, associated to the corresponding Kac-Moody algebra, which may be eventually also trivial.
In the last three decades, using the natural formalism of VOAs \cite{DM04b, EMS20, ELMS21}, this problem has been mostly solved by a case-by-case analysis, see \cite{LS19} (and \cite{LS20}) for a review, and through uniform approaches slightly later, see \cite{Hoe, Lam20} and \cite{HM22,MS23,MS24}, see also \cite{CLM22, LM23, BLS23, BLS24}. Indeed, we know that there exists exactly 70 strongly rational holomorphic VOAs with central charge $c=24$ and non-trivial weight-$1$ subspace. The remaining case is covered by the famous \textit{moonshine VOA} \cite{FLM88, Miy04}, whose uniqueness is still a big open problem. 

In the graded-local case, the central charge of a holomorphic VOSA can take any value in $\half\Zplus$. For example, the graded tensor product $F^l$ of $l\in\Zplus$ copies of the real free fermion VOSA $F$ is holomorphic with $c=\frac{l}{2}$. 
In a recent work \cite{HM}, cf.\ \cite{Ray24, BLTZ24}, the nice holomorphic VOSAs with central charge $c\leq 24$ and with non-trivial odd part where completely classified, \textit{up to the shorter moonshine $VB^\natural$ uniqueness conjecture} \cite{Hoe95}, stating that $VB^\natural$ is the unique nice holomorphic VOSA with central charge $c=\frac{47}{2}$ and vanishing weight-$\half$ and weight-$1$ subspaces (this is anyway implied by the moonshine uniqueness conjecture in \cite{FLM88}, see the discussion after \cite[Theorem 8.3]{HM} for details).
Their classification particularly relies on three ingredients, which we need to prove our results.
The first one is the \textit{Free-Fermion Splitting}, stating that for all $c\in\half\Zpluseq$ and all $l\in\Zpluseq$, the map $V\mapsto V\hat{\otimes}F^l$ defines a bijection between isomorphism classes of nice holomorphic VOSAs with central charge $c$ and with weight-$\half$ subspace of dimension $k\in\Zpluseq$ and isomorphism classes of nice holomorphic VOSAs with central charge $c+\frac{l}{2}$ and with weight-$\half$ subspace of dimension $k+l$.
This allows to reduce the classification problem to the one of nice holomorphic VOSAs with central charge $c=24$ only, see \cite[Section 8.5]{HM}. 
The second ingredient is the \textit{Neighbourhood Graph Method}, which states that the even part $V_\parzero$ of any nice holomorphic VOSA $V$ with $c=24$ and odd part $V_\parone\not=\{0\}$ is the fixed-point subalgebra of a strongly rational holomorphic VOA $U$ with the same central charge by an automorphism $g$ of order $2$ and of type $0$, which in turn gives rise to a (possibly isomorphic) strongly rational holomorphic VOA $U^{\mathrm{orb}(g)}$ with the same central charge by an \textit{orbifold construction} \cite{EMS20}.
The last one is a description of $V$, whenever the weight-$1$ subspace $V_1$ is non-trivial, as \textit{simple current extension} of a \textit{dual pair} $(W, V_K)$ of $V_\parzero$, where $K$ is the \textit{associated lattice} of $V_\parzero$ and $W$ is the \textit{coset subalgebra} of the lattice subalgebra $V_K$ in $V_\parzero$.
This last point allows to organize nice holomorphic VOSAs with central charge $c=24$ and with non-trivial weight-$1$ subspace into three \textit{glueing types}, depending on how the simple currents of $W\otimes V_K$ glue together to form $U, U^{\mathrm{org}(g)}$ and $V$. Moreover, there are two \textit{non-typical} VOSAs of glueing type III: the shorter moonshine VOSA tensorized with a real free fermion $VB^\natural \hat{\otimes}F$ and the \textit{odd moonshine VOSA} $VO^\natural$ \cite{DGH88, Hua96}, cf.\ \cite{GJ22, MS}. 
Of course, they fall into the case where $V_1=\{0\}$ and the shorter moonshine uniqueness conjecture prevents us to have a full classification result. 
To sum up, excluding hypothetical \textit{fake} copies of $VB^\natural \hat{\otimes}F$ (hypothetical \textit{fake} copies of $VO^\natural$ can be ruled out, see \cite[footnote 13]{HM} for details), there are exactly \textit{969} nice holomorphic VOSAs with central charge $c=24$ and with non-trivial odd part, see \cite[Theorem 8.3]{HM}.

There are two main results in the present paper: apart from hypothetical \textit{fake} copies of $VB^\natural$, we prove \textit{the unitarity of all nice holomorphic VOSAs with central charge at most $24$ and with non-trivial odd part} in Section \ref{subsec:unitarity} and \textit{their strong graded locality, except for those with central charge $24$ of glueing type III and with trivial weight-$\half$ subspace} in Section \ref{subsec:gl_conformal_nets}.

The unitarity result is obtained first for the case $c=24$ in Theorem \ref{theo:unitarity_holo_VOSAs_cc_24} and then for the case $c<24$ in Corollary \ref{cor:unitarity_holo_c<24} thanks to the Free-Fermion Splitting recalled above. We birefly review the idea of the proof of the former. Let $V$ be a nice holomorphic VOSA with central charge $c=24$ and with $V_\parone\not=\{0\}\not=V_1$. The key point is to prove the complete unitarity of the tensor product $W\otimes V_K$ of the dual pair $(W,V_K)$ of $V_\parzero$, so that one can conclude the unitarity of $V$ by the results in \cite{CGGH23, CGH} about the unitarity of VOSA extensions. 
We reach this goal by passing through the complete unitarity of $W$. It is proved by the remarkable result \cite[Theorem 8.6]{Gui24}, which implies the complete unitarity of strongly unitary and strongly rational VOAs with \textit{pointed} modular tensor categories of modules, as it is the case for $W$. Therefore, we also prove the strong unitarity of $W$ thanks to: the unitarity of irreducible twisted modules of lattice VOAs obtained from standard lifts of lattice isometries \cite[Section 3.2]{Lam23}; a similar application of \cite[Theorem 8.6, see Proposition 10.9]{Gui24} to the complete unitarity of fixed-point subalgebras of holomorphic VOAs; other than a crucial result \cite[Theorem 8.2]{HM} used for the classification. The unitarity of $VO^\natural$ is proved with similar tools, whereas the one of $VB^\natural\hat{\otimes} F$ was already known from \cite[Theorem 7.24]{CGH}. 

As for the unitarity, the result about the strong graded locality is proved first for the case $c=24$ in Theorem \ref{theo:strong_graded_locality_holo} and then for the case $c<24$ in Corollary \ref{cor:strong_graded_locality_c_less_24} thanks to the Free-Fermion Splitting.
It is based on Proposition \ref{prop:holo_vosas_strong_braiding}, where we prove the \textit{complete energy boundedness}, the \textit{strong intertwining property} and the \textit{strong braiding} for the \textit{unitary} intertwining operators of the tensor product $E\otimes X$ of a suitable dual pair $(E,X)$ in $V_\parzero$. The first of them implies the energy boundedness of $V$, which is a key property included in the definition of strong graded locality. The last two properties can be understood as a ``piecewise'' strong graded commutativity for the vertex operators of $V$, which is a simple current extension of $E\otimes X$, see Lemma \ref{lem:strong_graded_locality_simple_current_ext}.
These three fundamental properties are proved thanks to the fact that one of the factor, say $X$, in the tensor product $E\otimes X$ is already known to satisfy those desired properties; whereas the other one, that is $E$, forms a dual pair with a VOA $X'$ (possibly equal to $X$), still known to satisfy those properties, in at least one of the strongly rational holomorphic VOAs with $c=24$, obtained from the Neighbourhood Graph Method. 
The latter VOAs are proved to be strongly local in \cite[Theorem 5.5]{CGGH23} and thus we are allowed to use the big machinery developed in \cite{Gui21, Gui}, specifically in its part for coset models, see Theorem \ref{theo:strong_braiding_complete_energy_boundedness}. 
In total, without counting hypothetical \textit{fake} copies of $VB^\natural\hat{\otimes}F$, we are able to prove that 910 of the 969 nice holomorphic VOSAs with central charge $c=24$ and with non-trivial odd part are strongly graded-local, see \cite[Table 9]{HM}. 

Various preliminaries are presented in Section \ref{sec:preliminaries}. 
In Section \ref{subsec:preliminaries_unitary_VOSAs}, we recall the basic properties of strongly rational VO(S)As and of their representation theory. Then we discuss the condition of complete unitarity for VOAs and its implications for the unitary aspects of their representation theory. Finally, we treat the unitarity of CFT type VOSA extensions.
In Section \ref{subsec:preliminaries_VOSAs_gl_conformal_nets}, which serves Section \ref{subsec:gl_conformal_nets} only, we recall the correspondence between VOSAs and graded-local conformal nets from \cite{CKLW18, CGH} and also some facts about the representation theory of the latter.
Furthermore, we recall the main results from \cite{Gui21, Gui} and we use that setting to derive Lemma \ref{lem:strong_graded_locality_simple_current_ext}.

In conclusion, we end Section \ref{subsec:gl_conformal_nets} conjecturing that all nice holomorphic VOSAs with central charge $c\leq 24$ are strongly graded-local. A further interesting problem would be to determine which VOSAs show some kind of \textit{unitary superconformal} structures, which in turn will produce \textit{superconformal nets}, see \cite[Section 7]{CGH}.

\section{Preliminaries}  \label{sec:preliminaries}

\subsection{Completely unitary VOAs and VOSA extensions}
\label{subsec:preliminaries_unitary_VOSAs}

In the present section we set the notation and we recall the basic definitions and results which will be used throughout the paper.
Useful, but not exhaustive, references are: \cite{Kac01}, see also \cite{Xu98}, for the basic theory of \textit{vertex operator superalgebras (VOSAs)}; \cite{FLM88, FHL93, LL04} for the basic theory of \textit{vertex operator algebras (VOAs)} and their \textit{representation theory}; \cite{DL14, AL17}, \cite[Section 2.1]{Ten19b} for the \textit{unitary} theory for VOSAs and their modules. We will use the approach and the notation as in \cite[Section 4 and Section 5]{CKLW18} and in \cite[Section 3]{CGH}.

Let $\parzero$ and $\parone$ be the two equivalence classes in the cyclic group of order two $\Z_2:=\Z/2\Z$. For a VOSA $V$, we denote by $V_\parzero$ and $V_\parone$ its \textit{even} and \textit{odd} parts respectively, that is the eigenspaces of eigenvalues $1$ and $-1$ of a VOSA automorphism $\Gamma_V$ such that $\Gamma_V^{-1}=\Gamma_V$, called the \textit{parity operator}. We also define the vector superspace automorphism $Z_V:=(1+i)^{-1}(1_V+i\Gamma_V)$ of the vector superspace $V$. 
A vector $a\in V_\parzero$ is said to be \textit{even} with \textit{parity} $p(a)=\parzero$, whereas $b\in V_\parone$ is said to be \textit{odd} with \textit{parity} $p(b)=\parone$.
The \textit{vacuum} and the \textit{conformal} vectors of $V$ are even vectors denoted by $\Omega$ and $\nu$ respectively. For any vector $a\in V$, we use $Y(a,z)=\sum_{n\in \Z}a_{(n)}z^{-n-1}$ for its corresponding \textit{vertex operator}; in particular we write $Y(\nu,z)=\sum_{n\in \Z}L_nz^{-n-2}$. 
Recall that the set of operators $L_n$ for all $n\in\Z$ gives rise to a representation of the \textit{Virasoro algebra} with a certain \textit{central charge} $c\in \C$ on $V$, so that $V$ is said to be a VOSA with central charge $c$.
As usual, we say that $a\in V$ is a \textit{homogeneous vector} of \textit{conformal weight} $d_a\in \half\Z$ if $L_0a=d_a a$. 
Accordingly, $V_n$ is the eigenspace of $L_0$ of eigenvalue $n$ and we require the \textit{correct spin and statistics connection} for $V$, that is $V_\parzero=\bigoplus_{n\in\Z} V_n$ and $V_\parone=\bigoplus_{n\in\Z-\half}V_n$. 
Moreover, homogeneous vectors are said to be \textit{primary} if they are in the kernel of the operator $L_n$ for all $n\in\Zplus$ and \textit{quasi-primary} if they are in the kernel of $L_1$.
Recall that $\Omega$ is primary of conformal weight $0$, whereas $\nu$ is quasi-primary (but if $c=0$, then it is also primary) of conformal weight $2$.
We will use the notation $Y(a,z)=\sum_{n\in \Z-d_a}a_nz^{-n-d_a}$, where $a$ is any homogeneous vector and $a_n:=a_{(n+d_a-1)}$.

Now we turn to the representation theory of VOSA, first recalling some preliminary definitions. Let $V$ be any VOSA.
$V$ is said to be \textit{simple} if it has no non-trivial \textit{ideals} and of \textit{CFT type} if $V_0=\C\Omega$ and $V_n=\{0\}$ for all $n<0$. 
We follow the notions of (\textit{irreducible}) \textit{weak}, \textit{admissible} and \textit{ordinary} \textit{twisted} and \textit{untwisted} $V$-\textit{modules} as appearing in e.g.\ \cite[Section 2]{DNR21}; \textit{conformal weights} and \textit{homogeneous vectors} for $V$-modules are equivalently defined. In this paper, $V$-modules are meant to be ordinary and untwisted, unless differently specified.
For any vector $a\in V$ and any $V$-module $M$, we use the symbol $Y^M(a,z)=\sum_{n\in\Z}a_{(n)}^Mz^{-n-1}$ for the corresponding vertex operator; in particular we write $Y^M(\nu,z)=\sum_{n\in\Z}L_n^Mz^{-n-2}$.
Then, $V$ is said to be \textit{self-contragredient} if it is isomorphic to the \textit{contragredient module} $V'$ of its \textit{adjoint module}, that is $V$ itself seen as $V$-module. 
Furthermore, $V$ is said to be \textit{$C_2$-cofiniteness} if the vector space $V/\langle v_{(-2)}u\mid u,v\in V\rangle$ is of finite dimension. $V$ is said \textit{rational}, resp.\ \textit{regular}, if every \textit{admissible}, resp.\ \textit{weak}, $V$-module is a direct sum of \textit{irreducible} ordinary $V$-modules. By \cite[Theorem 6.6]{DZ06}, if $V$ is rational, then it has a finite number of irreducible $V$-modules.
A VOSA of CFT type is rational and $C_2$-cofinite if and only if it is regular, see \cite{Li99, ABD04} and \cite{HA15}, \cite[Section 2]{DNR21}.
A rational, $C_2$-cofiniteness and self-contragredient VOSA $V$ of CFT type (and thus also regular) is said to be \textit{strongly rational}. Recall that $V$ is self-contragredient if and only if it has a \textit{non-degenerate invariant bilinear form}. Then, a self-contragredient VOSA of CFT type is automatically simple by \cite[Proposition 3.8 (iv)]{CGH}, so that strongly rational VOSA are simple.

For a VOA $V$, we use the notation $\mathcal{V}\binom{K}{M\,\,\, N}$ for the vector space of \textit{intertwining operators of type} $\binom{K}{M\,\,\, N}$. 
One can refer to e.g.\ \cite[Section 5.4]{FHL93}, with the only difference that we allow real powers in the formal series of intertwining operators as in \cite[Section 1.3]{Gui19I}.
We also recall that intertwining operators among irreducible $V$-modules are said to be \textit{irreducible} as well. 
By the series of works \cite{HL95I, HL95II, HL95III, Hua95, Hua08}, it is known that the category $\Rep(V)$ of $V$-modules for a strongly rational VOA $V$ has a structure of \textit{modular tensor category} (sometimes also called modular category or modular fusion category), see e.g.\ \cite{BK01, EGNO15, Tur10}.
For a sketch of the construction of the modular tensor structure of $\Rep(V)$, one can refer to \cite[Secton 2.4]{Gui19I} or \cite[Section 4.1]{Gui21}.
Objects of $\Rep(V)$ are $V$-modules and arrows are homomorphisms of $V$-modules.
With an abuse of notation, we will use $M\in \Rep(V)$ to denote the picking of a $V$-module $M$ from objects of $\Rep(V)$. 
Of course, the additive structure of $\Rep(V)$ is given by the direct sum of $V$-modules, so that the simple objects are exactly the irreducible $V$-modules; whereas the \textit{tensor bifunctor} of $\Rep(V)$ is usually denoted by $\boxtimes$ and we use the symbol $b_{M,N}$ for the \textit{braiding} of $M,N\in\Rep(V)$.
For further use, we recall the definition of the tensor product $M\boxtimes N$ and of the braiding $b_{M,N}$ of two $V$-modules $M$ and $N$. As $V$ is rational, the set $\{M_k\mid k\in\mathcal{E}\}$ with $\mathcal{E}$ a set of indexes, constituted by a representing module $M_k$ for every equivalence class of irreducible $V$-modules, is finite. We let $V$ be in this set.
Then  
\begin{equation}   \label{eq:defin_cat_tensor_product}
	\begin{split}
		M\boxtimes N
			&:=
		\bigoplus_{k\in\mathcal{E}}\mathcal{V}\binom{M_k}{M \,\,\, N}^*\otimes M_k   \\
		Y^{M\boxtimes N}(a,z)
			&:=
		\bigoplus_{k\in\mathcal{E}}1_{\mathcal{V}\binom{M_k}{M \,\,\, N}^*}\otimes Y^{M_k}(a,z)
		\qquad \forall a\in V  
	\end{split}
\end{equation}
where $\mathcal{V}\binom{M_k}{M \,\,\, N}^*$ is the dual of the vector space $\mathcal{V}\binom{M_k}{M \,\,\, N}$, which is finite dimensional thanks to the rationality of $V$.
Then a tensor product of morphisms can be defined accordingly, see the references given above.
For any $K\in\Rep(V)$ and any intertwining operator $\Y$ of type $\binom{K}{M\,\,\, N}$, define an intertwining operator $B_+\Y$ of type $\binom{K}{N\,\,\, M}$ by
\begin{equation}  \label{eq:braided_intertwining_operator}
	(B_+\Y)(v,z)w:=e^{zL_{-1}^K}\Y(w,-z)v 
	\qquad\forall v\in N \,\,\, \forall w\in M \,.
\end{equation}
$B_+\Y$ is called a \textit{braided intertwining operator}.
Therefore for all $k\in\mathcal{E}$, $B_+$ induces a linear operator from $\mathcal{V}\binom{M_k}{M\,\,\, N}$ to $\mathcal{V}\binom{M_k}{N\,\,\, M}$, so that if $\sigma_{M,N}^k$ is its transpose from $\mathcal{V}\binom{M_k}{M\,\,\, N}^*$ to $\mathcal{V}\binom{M_k}{N\,\,\, M}^*$, we can define the morphism:
\begin{equation}   \label{eq:defin_braiding}
	b_{M,N}:=\sum_{k\in\mathcal{E}}\sigma_{M,N}^k\otimes \one_{M_k}:M\boxtimes N\to N\boxtimes M
\end{equation}
which gives the desired \textit{braid operator} of $\Rep(V)$.
The \textit{identity object} of $\Rep(V)$ is the adjoint module, still denoted by $V$, and the \textit{categorical dual} of $M\in\Rep(V)$ is given by the contragredient module $M'$, so that $M\boxtimes M'\cong V\cong M'\boxtimes M$. 
We also recall that the \textit{twist} $\omega(M)$ for any $M\in\Rep(V)$ is defined as the operator $e^{i2\pi L_0^M}$.

\begin{defin}  \label{defin:simple_currents}
	Let $V$ be a strongly rational VOA. Any $V$-module $M\in\Rep(V)$ is said a \textit{simple current} if it is invertible, that is there exists a $V$-module $X\in\Rep(V)$ such that $M\boxtimes X\cong V$. If $M\not\cong V$ and there exists $n\in \Zplus$ such that $M^{\boxtimes n}\cong V$ and $M^l\not\cong V$ for all integers $1<l< n$, then $M$ is called a \textit{$\Z_n$-simple current}.
	$\Rep(V)$ is \textit{pointed} if all simple objects are invertible, that is simple currents.
\end{defin}

Note that the $V$-module $X$ in Definition \ref{defin:simple_currents} must be equivalent to $M'$, so that $X\boxtimes M\cong V$. 
For all simple currents $M,N\in\Rep(V)$, we have that $M$, $N$ and $M\boxtimes N$ are all irreducible.
Moreover, the categorical dimension $d(M)$ (see e.g.\ \cite[Definition 4.7.11]{EGNO15}) of any invertible object $M$ is equal to $\pm 1$, see e.g.\ \cite[Remark 2.3]{FRS04}. Then it follows from \cite[Remark 2.33]{DGNO10} that there exists a complex number $\alpha(M)$ such that $\omega(M)=\alpha(M)\one_M$ and $b_{M,M}=\alpha(M)d(M)\one_{M\boxtimes M}$. In particular, if $\Rep(V)$ is \textit{even}, that is if the categorical dimensions of all simple objects are positive, then $d(M)=1$, $\omega(M)=\alpha(M)\one_M$ and $b_{M,M}=\alpha(M)\one_{M\boxtimes M}$. Then under the following condition on $V$, $\Rep(V)$ is even as the quantum dimension of every $V$-module \cite{DJX13}, which is a positive number, equals the categorical one by \cite[Proposition 3.11]{DLN15}. 

\begin{defin}
	 A VOSA $V$ is said to be \textit{nice} if it is strongly rational and it satisfies the \textit{positivity condition}, that is every irreducible untwisted or $\Gamma_V$-twisted $V$-module $M$ different from the adjoint module $V$ has positive $L_0^M$-grading.
\end{defin}

Note that the positivity condition for a VOA $V$, that is a VOSA $V$ with trivial grading $\Gamma_V=1_V$, reduces to a positivity condition for the irreducible (untwisted) $V$-modules only. 

\begin{rem}  \label{rem:niceness_even_odd_parts}
	A simple VOSA $V$ of CFT type is nice if and only if $V_\parzero$ is.
	Indeed, if $V$ is strongly rational, then $V_\parzero$ is too by \cite[Proposition 2.4]{HM} (there it is used that $V_\parzero$ is regular if and only if $V$ is as claimed at the beginning of \cite[Section 4]{DNR21}, see also the proof of \cite[Proposition 3.4]{DRY22}, based on \cite{Miy15, CM}); if $V$ is nice, then $V_\parzero$ is by \cite[Proposition 2.7]{HM}, based on \cite[Proposition 5.8]{DRY22}. 
	Vice versa, if $V$ is a simple VOSA, then $V_\parzero$ is a simple VOA and $V_\parone$ is an irreducible $V_\parzero$-module, see e.g.\ \cite[Remark 4.17]{CGGH23}. If $V$ is also of CFT type, then $V_\parzero$ is too and $V_\parone$ has necessarily positive $L_0$-grading. Even more, if $V_\parzero$ is strongly rational, then $V_\parone$ is known to be a $\Z_2$-simple current of $V_\parzero$, see e.g.\ \cite[Theorem 3.1 and Remark A.2]{CKLR19}. By \cite[Proposition 2.6]{HM}, we can conclude that $V$ is strongly rational and the niceness follows from the fact that every untwisted and $\Gamma_V$-twisted $V$-module is an untwisted $V_\parzero$-module.
\end{rem}

\begin{rem}
	Note that by \cite[Theorem 3.9]{CKL20}, see also \cite[Theorem 4.2]{CKM24}, if $V$ is a strongly rational VOA such that $\Rep(V)$ is even (for example when $V$ is nice), then any extension of $V$ by a $\Z_2$-simple current $M$ with twist $\omega(M)=-\one_M$ (or equivalently with braiding $b_{M,M}=-\one_{M\boxtimes M}$), cf.\ Definition \ref{defin:vosa_extensions} below, will produce a VOSA (which must have the correct spin and statistics connection by our definition).
\end{rem}

\begin{rem}  \label{rem:discriminant_form}
	The ribbon equivalence classes of pointed modular tensor categories are exhausted by the categories of type $\mathcal{C}(A,q)$ associated to the metric group $(A,q)$, where $A$ is a finite abelian group and $q:A\to \C/\Z$ is a non-degenerate quadratic form, see \cite[Theorem 3.3]{JS93}, see also \cite[Section 8.4 and Example 8.13.5]{EGNO15}. 
	Simple objects of $\mathcal{C}(A,q)$ are given by $X^\alpha$ for all $\alpha\in A$, satisfying $X^\alpha\boxtimes X^\beta=X^{\alpha+\beta}$ for all $\alpha,\beta\in A$.
	Then the braid operator $b_{X^\alpha,X^\beta}$ is given by $q(\alpha+\beta)-q(\alpha)-q(\beta)$ and the twist $\omega(X^\alpha)$ is $e^{i2\pi q(\alpha)}\one_{X^\alpha}$. 
	Also the modular matrices $S$ and $T$ are described in terms of $q$.
	If $V$ is a nice VOA with pointed even modular tensor category $\Rep(V)$, then we can realize the latter as $\mathcal{C}(A_V,q_V)$ for some metric group $(A_V, q_V)$. Moreover, as the conformal weights of $V$-modules are all rational numbers \cite[Theorem 1.1]{DLM00}, cf.\ also \cite[Theorem 8.9]{DZ05} and \cite[Theorem 2.1]{DNR21}, we have that $q_V$ actually takes values in $\Q/\Z$ and we call $(A_V,q_V)$ a \textit{discriminant form}, see e.g.\ \cite{Nik80}. With an abuse of notation, we may identify $A_V$ with $(A_V,q_V)$.
\end{rem}

\begin{defin}
	 A VOSA $V$ is said to be \textit{holomorphic} (or self-dual) if the adjoint module is the unique irreducible $V$-module.
\end{defin}

Now, we move to the unitarity for VOSAs and their modules: 

\begin{defin}
	A VOSA $V$ is said to be \textit{unitary} if the vector superspace $V$ is equipped with a \textit{scalar product}, linear in the second variable, and a VOSA automorphism $\theta$, called the \textit{PCT-operator}, such that $\scalar$ is \textit{normalized} and $\theta$ is an \textit{involution}, that is $(\Omega|\Omega)=1$ and $\theta^{-1}=\theta$ respectively. Furthermore, they must satisfy the following \textit{invariant property}:
	\begin{equation} \label{eq:invariant_property}
		(Y(\theta(a),z)b|c)=(b|Y(e^{zL_1}(-1)^{2L_0^2+L_0}z^{-2L_0}a, z^{-1})c)
		\qquad\forall a,b,c\in V \,.
	\end{equation}
	Similarly, an untwisted or twisted $V$-module $M$ is said to be \textit{unitary} if it has a scalar product such that the vertex operators $Y^M(a,z)$ for all $a\in V$ satisfy the invariant property \eqref{eq:invariant_property} with the same PCT operator of $V$ and for all $b,c\in M$. Accordingly, any untwisted or twisted $V$-module is said \textit{unitarizable} if it can be equipped with a unitary structure. 
\end{defin}

\begin{rem} 
	Recall that a unitary VOSA is automatically self-contragredient and it is simple if and only if it is of CFT type, see \cite[Proposition 3.10]{CGH} and its proof. 
\end{rem}

\begin{rem} \label{rem:unitarity_even_odd_parts}
		By \cite[Theorem 3.11, cf.\ Remark 3.12]{CGH}, if $V$ is a simple VOSA, then $V$ is unitary if and only if $V_\parzero$ is a unitary VOA and $V_\parone$ is a unitary $V_\parzero$-module. In both cases, $V_\parzero$ and $V_\parone$ are automatically simple and irreducible respectively, see Remark \ref{rem:niceness_even_odd_parts}.
\end{rem}

\begin{rem}  \label{rem:equiv_defin_unitarity}
	One may use the following alternative definition of unitarity, cf.\ \cite[Section 2]{CT23}: a VOSA $V$ is said unitary if there exist a normalized scalar product $\scalar$ and an antilinear vector superspace involution $V\ni a\to \overline{a}\in V$ such that $\overline{\nu}=\nu$ and
	\begin{equation}
		(a_nb| c)=(b|\overline{a}_{-n}c) \qquad\forall a,b,c\in V \,.
	\end{equation}
	Note that $\overline{\Omega}=\Omega$. The equivalence of these two definitions of unitarity follows from \cite[Theorem 3.31]{CGH}, giving us the following relation between the involution $\overline{\cdot}$ and the PCT operator $\theta$: $\overline{a}=e^{L_1}(-1)^{2L_0^2+L_0}\theta(a)$ for all $a\in V$. Accordingly, a homogeneous vector $a\in V$ is said \textit{Hermitian} if $\overline{a}=a$.
\end{rem}

Following \cite[Section 2]{Gui}, we call \textit{unitary} the intertwining operators among unitary modules of a unitary VOA.
Let $\Repu(V)$ be the category of \textit{unitary} $V$-modules, that is objects of $\Repu(V)$ are unitarizable $V$-modules equipped with fixed unitary structures.
We let the adjoint module $V$ be in $\Repu(V)$ with the unitary structure of the unitary VOA $V$. In the following, we may identify $V$ with its contragredient module $V'$.
Therefore, $\Repu(V)$ is naturally a $C^*$-category, where the $^*$-structure is given by taking for a morphism $T$ between unitary $V$-modules $M$ and $N$, its adjoint $T^*$ from $N$ to $M$ with respect to the scalar products on these $V$-modules. 

Suppose that $V$ is strongly rational. 
In \cite{Gui19I, Gui19II}, see also \cite[Section 4.3]{Gui21} and \cite[Section 1.3]{Gui22}, for all $M, N\in\Repu(V)$, the author defines a non-degenerate invariant Hermitian form $\Lambda_{M,N}\scalar$ for the dual vector spaces of intertwining operators appearing in the definition \eqref{eq:defin_cat_tensor_product} of $M\boxtimes N$. One can show that if $\Lambda_{M,N}\scalar$ is \textit{positive-definite}, then it can be extended to an invariant scalar product on $M\boxtimes N$, making the structural isomorphisms of $\Rep(V)$ unitary. 
Then $\Lambda_{M,N}\scalar$ is used to define a categorical tensor product on $\Repu(V)$, turning it into a unitary ribbon fusion category.
In the following, we will denote $\Lambda_{M,N}\scalar$ and its extension to any product $M\boxtimes N$ by the same symbol.
Then it comes natural to have the following definitions:

\begin{defin}[\cite{Ten18}]
	A simple rational VOA $V$ is said to be \textit{strongly unitary} if it is unitary and all $V$-modules are unitarizable. 
\end{defin}

\begin{defin}[{\cite[Definition 8.1 and Definition 10.1]{Gui24}, cf.\ \cite[Definition 1.8]{Gui22}}]
	Let $V$ be a unitary strongly rational VOA. If $M, N\in \Repu(V)$ are such that $\Lambda_{M,N}\scalar$ is positive-definite, then $M\boxtimes N$ is said to be (algebrically) \textit{positive}. 
	Therefore $V$ is said to be \textit{completely unitary} if it is strongly unitary and $M\boxtimes N$ is positive for all irreducible unitary $V$-modules $M$ and $N$.
\end{defin}

In the completely unitary case, $\Repu(V)$ is a \textit{unitary modular tensor category}, so that for all unitary $V$-modules $M,N$ and $K$, we will safely identify $M\boxtimes(N\boxtimes K)$ with $(M\boxtimes N)\boxtimes K$, just writing $M\boxtimes N\boxtimes K$, and both $V\boxtimes M$ and $M\boxtimes V$ with $M$.  

\begin{rem}  \label{rem:niceness_completely_unitary_voas}
	Irreducible unitary modules of simple unitary VOAs have positive $L_0$-grading, apart from the adjoint module, see \cite[Proposition 4.5]{CGGH23}, cf.\ also \cite[Proposition 1.7]{Gui19I}. Therefore, strongly rational and strongly unitary VOAs (and thus also completely unitary VOAs) are nice. 
\end{rem}

The following remarkable results are crucial for our goals.

\begin{theo}[{\cite[Theorem 8.6]{Gui24}}] \label{theo:positivity_fusion_product}
	Let $V$ be a strongly rational VOA with $M$ and $N$ be irreducible unitary $V$-modules. If $M\boxtimes N$ is an irreducible unitarizable $V$-module, then $M\boxtimes N$ is positive.
\end{theo} 

\begin{cor}  \label{cor:complete_unitarity_pointed_mtc}
	Strongly unitary and strongly rational VOAs with pointed modular tensor categories of modules are completely unitary.
\end{cor}

A \textit{vertex subalgebra} $W$ of a VOSA $V$ is a subspace of $V$ which is also closed with respect to the $(n)$-product and containing the vacuum vector: in particular $W$ inherits from $V$ the vertex superalgebra structure, see e.g.\ \cite[Section 4.3]{Kac01}. 
Moreover, for a subset $S$ of $V$, we denote by $W(S)$ the smallest vertex subalgebra of $V$ containing $S$.
We say that $W$ is \textit{full} if it contains the conformal vector of $V$.
Recall that a vertex subalgebra is $L_{-1}$-invariant. If $V$ is a unitary VOSA, we call $W$ a \textit{unitary (vertex) subalgebra} if it is also $L_1$-invariant and PCT-invariant, see \cite[Proposition 3.33]{CGH}. In this case, see \cite[Proposition 3.35]{CGH}, if $V$ is also simple, then we have that $W$ is a simple unitary VOSA: the conformal vector is given by $\nu^W:=e_W\nu$ with $e_W$ the projection of $V$ onto $W$, and the operators $L_n$ and $L^W_n$ coincide on $W$ for all $n\in\{-1,0,1\}$; the invariant scalar product and the PCT operator of $W$ are inherited by restriction from the ones of $V$. 

\begin{rem}  \label{rem:fixed-point_subalgebras}
	Let $V$ be a unitary VOSA. We call $\Aut(V)$ and $\Aut_{\scalar}(V)$ the groups of \textit{(VOSA) automorphisms} and \textit{unitary (VOSA) automorphisms} of $V$ respectively, see \cite[Section 3.3]{CGH}.
	If $G\subseteq\Aut(V)$ is a closed subgroup, then the \textit{fixed-point subalgebra} 
	$$
	V^G:=\{a\in V\mid g(a)=a \,\,\,\forall g\in G\}
	$$ 
	is a vertex subalgebra of $V$. Moreover, if $G\subseteq\Aut_{\scalar}(V)$, then $V^G$ is also a unitary subalgebra of $V$, see \cite[Example 3.34]{CGH}.
	Note that the even part $V_\parzero$ is equal to $V^{\{1_V, \Gamma_V\}}$ and indeed it is a full unitary subalgebra of $V$.
\end{rem}

Fixed-point subalgebras by finite cyclic subgroups of automorphisms of strongly rational holomorphic VOAs have pointed modular tensor categories of modules by \cite[Theorem 5.2 and Proposition 5.6]{EMS20}. Therefore, Corollary \ref{cor:complete_unitarity_pointed_mtc} applies to this case: 

\begin{theo}[{\cite[Theorem 10.4 and Proposition 10.9]{Gui24}}]  \label{theo:complete_unitarity_holo_orbifold}
	Let $V$ be a unitary strongly rational holomorphic VOA and $G$ be a finite cyclic subgroup of $\Aut_{\scalar}(V)$. Suppose that all irreducible $G$-twisted $V$-module are unitarizable. Then $V^G$ is a completely unitary VOA.
\end{theo}

Now we move to the \textit{graded tensor product} $V^1\hat{\otimes} V^2$ of two VOSAs $V^1$ and $V^2$, see e.g.\ \cite[Section 4.3]{Kac01}, \cite[Theorem 3.2.4]{Xu98}. As vector superspace we have the usual tensor product of vector superspaces $V^1\otimes V^2$ with parity operator $\Gamma_{V^1\hat{\otimes} V^2}=\Gamma_{V^1}\otimes \Gamma_{V^2}$. 
With obvious notation, the vacuum and the conformal vectors are $\Omega^1\otimes \Omega^2$ and $\nu^1\otimes \Omega^2+\Omega^1\otimes\nu^2$ respectively. Moreover, the vertex operators are defined by
\begin{equation}  \label{eq:defin_graded_tensor_product_vertex_ops}
	Y_{V^1\hat{\otimes}V^2}(a^1\otimes a^2, z)
	:=
	Y_{V^1}(a^1, z)\Gamma_{V^1}^{p(a^2)}\otimes Y_{V^2}(a^2, z)
	\qquad\forall a^1\in V^1
	\,\,\,\forall a^2\in V^2 \,.
\end{equation}
If either $V^1$ or $V^2$ is a VOA, then $\Gamma_{V^1}^{p(a^2)}=1_{V^1}$ and accordingly we may use the symbol $\otimes$ in place of $\hat{\otimes}$ and call it the \textit{tensor product} of $V^1$ and $V^2$.

\begin{rem}  \label{rem:graded_tensor_product}
	We have that $V^1\hat{\otimes}V^2$ is simple if and only if $V^1$ and $V^2$ are, cf.\ the proof of \cite[Proposition 4.7.2]{FHL93}.
	It is easy to see that if $V^1$ and $V^2$ are CFT type, resp.\ self-contragredient, then $V^1\hat{\otimes}V^2$ is of CFT type, resp.\ self-contragredient. Recall that the tensor product of regular VOAs is regular by \cite[Proposition 3.3]{DLM97}, so that the tensor product of strongly rational VOAs is still strongly rational. 
	Furthermore, as for \cite[Proposition 4.7.2]{FHL93}, the irreducible modules of a tensor product of VOAs are tensor product of irreducible modules of those VOAs, if $V^1$ and $V^2$ are VOAs satisfying the positivity condition, then their tensor product does the same. 
	Note that if $V^1\hat{\otimes}V^2$ is a unitary VOSA, then $V_1$ and $V_2$ are unitary subalgebras of it. By \cite[Proposition 2.4]{AL17}, cf.\ also \cite[Proposition 2.20]{Ten19}, if $V^1$ and $V^2$ are unitary VOSAs, then $V^1\hat{\otimes}V^2$ is too.
	Moreover, by \cite[Proposition 3.31]{Gui22} and its proof, if $V^1$ and $V^2$ are regular VOAs of CFT type, then they are strongly unitary, resp.\ completely unitary, VOAs if and only if $V^1\otimes V^2$ is strongly unitary, resp.\ completely unitary. 
\end{rem}

\begin{rem}  \label{rem:coset_subalgebras}
	Let $W$ be a vertex subalgebra of a VOSA $V$. Then the \textit{coset} (or commutant) \textit{subalgebra} of $W$ in $V$, see e.g.\ \cite[Remark 4.6b]{Kac01}, cf.\ also \cite[Section 5]{FZ92} and \cite[Section 3.11]{LL04}, 
	$$
	\Com_V(W):=\{a\in V\mid [Y(a,z),Y(b,w)]=0 \,\,\forall b\in W\}
	=\left\{a\in V\mid b_{(j)}a=0 \,\,\substack{\forall b\in W \\ \forall j\in\Zpluseq}\right\}
	$$
	(in some references, it is simply denoted by $W^c$ when no confusion arises), is a vertex subalgebra of $V$. 
	By \cite[Theorem 5.1]{FZ92}, see also \cite[Theorem 3.11.12]{LL04}, if $V$ is a VOA of CFT type and $W$ is a vertex subalgebra, whose vertex algebra structure has a conformal vector $\nu^W\in V_2$, satisfying $L_1\nu^W=0$, then $\Com_V(W)$ is a VOA too with conformal vector $\nu^{\Com_V(W)}=\nu-\nu^W$ and the operators $L_n$ and $L^{W}_n$ coincide on $W$ for all $n\geq -1$.
	By \cite[Proposition 3.36]{CGH}, if $V$ is a simple unitary VOSA with $W$ a unitary subalgebra, then $\Com_V(W)$ is a unitary subalgebra too with conformal vector $\nu^{\Com_V(W)}=\nu-\nu^W$ and the operators $L_0^W$ and $L_0^{\Com_V(W)}$ are simultaneously diagonalizable on $V$ with non-negative eigenvalues.
\end{rem}

\begin{rem}  \label{rem:lattice_vosas}
	Important examples of unitary VOSAs used in this paper are \textit{lattice VOSAs}, see e.g.\ \cite[Section 5.5]{Kac01} and \cite[Section 6]{Xu98}. The reader can also refer to \cite[Section 4.4]{DL14} and \cite[Section 2.4]{AL17} for the basic unitary structure of these models. See also \cite{CGH19} for some example of unitary subalgebras of lattice VOAs. Recall that to every \textit{positive-definite integral lattice}, that is a free $\Z$-module $L$ of finite rank equipped with a scalar product $\pairing:L\times L\to \Z$, is associated a unique, up to isomorphism, simple unitary VOSA $V_L$. In particular, if $L$ is \textit{even}, that is $\langle \alpha, \alpha\rangle\in 2\Z$ for all $\alpha\in L$, then $V_L$ is a VOA; otherwise $L$ is \textit{odd} and $(V_L)_\parone\not=\{0\}$.
	Moreover, the even part $(V_L)_\parzero$ of a lattice VOSA $V_L$ is the lattice VOA $V_{L_\parzero}$ arising from the \textit{even sublattice} $L_\parzero:=\{\alpha\in L\mid \langle \alpha,\alpha\rangle \in 2\Z \}$ of $L$.
	All these models are nice, see \cite{DL93} and \cite[Theorem 3.16]{DLM97} for the even case, \cite{BK04} and \cite[Section 6.4]{Xu98} for the odd one. 
	If $L$ is even, then $\Rep(V_L)$ is a pointed even modular tensor category equivalent to the category $\mathcal{C}(L'/L, q_L)$ associated to the discriminant form $(L'/L, q_L)$, see Remark \ref{rem:discriminant_form}, where $L':=\{\beta\in \Q\otimes_\Z L\mid \langle \beta|\alpha\rangle\in \Z\,\,\, \forall \alpha\in L\}$ is the \textit{dual lattice} of $L$ and $q_L:L'/L\to \Q/\Z$ is defined by $q_L(\beta):=\frac{\langle \beta,\beta\rangle}{2} +\Z$.
	Similarly, if $L$ is odd, then the equivalence classes of irreducible $V_L$-modules are parametrized by the equivalence classes $\beta+L\in L'/L$ (where $L'$ is defined by the same formula) and they have conformal weights in $\frac{\langle \beta,\beta\rangle}{2}+\half\Z$ correspondingly. 
	Accordingly, if the lattice $L$ is \textit{unimodular}, that is $L=L'$, then $V_L$ is a nice holomorphic VOSA. 
	Lattice VOAs are also completely unitary by \cite[Theorem 5.8]{Gui21}.
\end{rem}

\begin{lem} \label{lem:unitary_subalgebra_lattice_vosas}
	Let $L$ be a positive-definite integral lattice and $K$ be a sublattice of $L$. Then $V_K$ is a unitary subalgebra of $V_L$.
\end{lem}

\begin{proof}  
	By \cite[Proposition 5.5]{Kac01}, a conformal vector for a lattice VOSA $V_Q$, for any integral lattice $Q$ with non-degenerate bilinear form, can be realized as the conformal vector of its Heisenberg subalgebra generated from the complexification $\C\otimes_\Z Q$ of $Q$. Then we can choose the conformal vector $\nu$ of $V_L$ as $\nu^K+ \nu^{K^\perp}$, where $\nu^K$ and $\nu^{K^\perp}$ are conformal vectors arising from $K$ and its orthogonal complement $K^\perp$ in $L$. 
	Letting act $Y(\nu,z)=\sum_{n\in\Z}L_nz^{-n-2}$ on the generators of $V_K$, it is not difficult to see that $V_K$ turns out to be $L_1$-invariant.
	
	By \cite[Theorem 2.9]{AL17}, $V_L$ is a unitary VOSA with PCT operator $\theta$ defined on the generators by:
	\begin{equation} \label{eq:pct_lattice_vosas}
		\theta(\alpha^{(1)}_{-n_1}\cdots\alpha^{(m)}_{-n_m}\Omega\otimes e^{\alpha} )
		= (-1)^m\alpha^{(1)}_{-n_1}\cdots\alpha^{(m)}_{-n_m}\Omega\otimes e^{-\alpha}
	\end{equation}
	where $m\in\Zpluseq$, $\alpha\in L$, $n_j\in\Zplus$ and $\alpha^{(j)}\in \C\otimes_\Z L$ for all $j\in\{1,\dots, m\}$.
	$V_K$ is clearly $\theta$-invariant by the action \eqref{eq:pct_lattice_vosas}.
	Therefore, $V_K$ turns out to be a unitary subalgebra of $V_L$ as desired. 
\end{proof}

\begin{defin}  \label{defin:vosa_extensions}
	A \textit{simple CFT type VO(S)A extension} $U$ of a simple VOA $V$ of CFT type is a simple VO(S)A of CFT type such that $V$ is a full vertex subalgebra of $U$. Similarly, $U$ is a \textit{unitary CFT type VO(S)A extension} of $V$ if $U$ is a unitary VO(S)A of CFT type (and thus also simple) and $V$ is a full unitary subalgebra of $U$. Accordingly, $U$ is called a \textit{(unitary/\textit{unitarizable}) simple current extension} of $V$ if $U$ is a (unitary/unitarizable) $V$-module which can be decomposed into (unitary/unitarizable) simple currents of $V$.
	A simple current extension $U$ of $V$ is said to be \textit{$G$-graded} for an abelian group $G$ if it decomposes as $\bigoplus_{\chi\in G}V^\chi$, where $V^{1_G}=V$ and $V^\chi\boxtimes V^\eta\cong V^{\chi\eta}$ for all $\chi,\eta\in G$.  
\end{defin}

Note that if $V$ is a strongly rational VOA, then the number of irreducible $V$-modules appearing in the simple current extension $U$ in the above definition is finite.

We recall that in the case of a nice VOA $V$, every isomorphism class of a simple CFT type VOA extension $U$ of $V$ corresponds bijectively to an equivalence class of a \textit{rigid commutative haploid algebra} $B$ with \textit{trivial twist} in $\Rep(V)$, see \cite[Theorem 3.6 and Remark 3.7]{HKL15}, cf.\ also \cite[Theorem 5.2]{KO02}. 
Moreover, $\Rep(U)$ is equivalent to the ribbon tensor category $\Mod_{\Rep(V)}^0(B)$ of \textit{local $B$-modules}, see \cite[Theorem 1.17]{KO02} (we are mostly using the notation and the terminology as in \cite[Section 2]{CGGH23}).
As any $U$-module is also a $V$-module and $V$ and $U$ share the same conformal vector, we have that $U$ satisfies the positivity condition because $V$ does it.
Therefore $U$ is also strongly rational because, as stated at the beginning of \cite[Section 3.5]{Gui22}, the argument used in the proof of \cite[Theorem 4.14]{McR20} easily generalizes to our case.
Therefore $\Rep(U)$ is a modular tensor category and thus $\Mod_{\Rep(V)}^0(B)$ is too.
Now, suppose that $V$ is also completely unitary.
In \cite[Theorem 2.21]{Gui22}, it is proved that every isomorphism class of a unitary CFT type VOA extension $U$ of $V$ corresponds bijectively to an equivalence class of a \textit{commutative haploid special $C^*$-Frobinius algebra} $B$, also called a \textit{commutative haploid Q-system}, with \textit{trivial twist} in $\Repu(V)$. Moreover, by \cite[Theorem 3.30]{Gui22}, $U$ is completely unitary and $\Repu(V)$ is equivalent to $\Mod_{\Rep(V)}^{u,0}(B)$, that is the category of \textit{unitary} local $B$-modules, as unitary modular tensor categories.
By \cite[Theorem 3.2]{CGGH23}, every haploid algebras in a $C^*$-tensor category, as $\Repu(V)$ is, is rigid if and only if it is equivalent to a Q-system in the same category. Therefore, we have:

\begin{theo}[{\cite[Theorem 4.7 and Corollary 4.18]{CGGH23}}] \label{theo:unitarity_vosa_extensions}
	Every simple CFT type VOSA extension $U$ of a completely unitary VOA $V$ is a unitary CFT type VOSA extension of $V$. In particular, $U_\parzero$ is a unitary CFT type VOA extension of $V$, which is also a completely unitary VOA.
\end{theo}

For the reader convenience, we write a proof of the following fact, which is probably already well-known to experts.

\begin{lem} \label{lem:simple_current_extensions}
		Let $V$ be a nice VOA and $U$ be any simple current extension of $V$. Then $U$ is $G$-graded for some finite abelian group $G$.
		In particular, if $\Rep(V)$ is pointed, then any simple CFT type VOSA extension of $V$ is a $G$-graded simple current extension for some finite abelian group $G$.
\end{lem}

\begin{proof}  
	As $V$ is a strongly rational VOA, then $\Rep(V)$ is a modular tensor category.
	Let $\mathcal{G}$ be the finite set of equivalence classes of the simple currents appearing in the decomposition of $U$, considered as a $V$-module, into irreducible $V$-modules. Then the existence of a finite abelian group $G$ is assured if we prove that $\mathcal{G}$ has an abelian group structure induced by $\boxtimes$.
	
	First of all, $V\in\mathcal{G}$ and it will be the identity element. 
	It is also known that $\mathcal{G}$ is closed under the categorical tensor product $\boxtimes$ and taking contragredient, cf.\ e.g.\ \cite[Remark 11.3]{Gui24}. 
	Indeed, for all (equivalence class representatives) $M,N\in\mathcal{G}$, the vertex operator $Y_U$ of $U$ restricts to an intertwining operator $Y_{M,N}$ of $V$ of type $\binom{U}{M\,\,\, N}$. Moreover, $Y_{M,N}$ is non-zero by the simplicity of $U$.
	Indeed, if $Y_{M,N}$ is trivial, then $M$ would be contained in the \textit{annihilator} of $N$ in $U$, that is $\mathcal{I}_U(N):=\{a\in U\mid Y(a,z)c=0\,\,\,\forall c\in N\}$. It is not difficult to see that this is an ideal of $U$, see e.g.\ the proof of \cite[Proposition 4.5.11]{LL04}. Then $\mathcal{I}_U(N)=U$, which implies that $N$ is an ideal of $U$, contradicting its simplicity. 
	Therefore $Y_{M,N}$ further restricts to an intertwining operator of type $\binom{K}{M\,\,\, N}$ for a simple current $K$ in $U$, which must necessarily be isomorphic to $M\boxtimes N$. Then $M\boxtimes N\in\mathcal{G}$.  
	Similarly, we have that for all $M\in \mathcal{G}$, $Y_U$ restricts to a non-zero intertwining operator of type $\binom{M}{V\,\,\, M}$. 
	As $V$ is nice, $L_1U_1=\{0\}$ and thus $U$ is a self-contragredient VOSA by \cite[Proposition 3.8]{CGH}.
	Then it is also self-contragredient as $V$-module, so that $M'\in\mathcal{G}$ for all $M\in\mathcal{G}$.
	Moreover, $\boxtimes$ is associative and commutative as operation on $\mathcal{G}$ thanks to the associator and the braiding of $\Rep(V)$ respectively.
	 
	To conclude, we only need to prove that every simple current $M\in \mathcal{G}$ appears with multiplicity one in the decomposition of $U$, cf.\ \cite[Remark 3.4(i)]{FRS04}.
	To prove it, one can adapt the first part of the proof of \cite[Proposition 4.23]{CGGH23}, see also the one of \cite[Theorem 7.25]{CGH}.
	Note that $U$ decomposes also into $U_\parzero\oplus U_\parone$ as $V$-module and thus if $M\in \mathcal{G}$, then either $M$ is a $V$-submodule of $U_\parzero$ or of $U_\parone$.
	Suppose that $M\subseteq U_\parone$; the case $M\subseteq U_\parzero$ can be done similarly.
	Let $B$ be the rigid haploid commutative algebra in $\mathcal{C}:=\Rep(V)$ giving the extension $V\subseteq U_\parzero$.
	As $V$ satisfies the positivity condition, then $U_\parzero$ does so, so that it is also nice. Then $\Rep(U_\parzero)$ and $\operatorname{Mod}^0_\mathcal{C}(B)$ are equivalent modular tensor categories, see the discussion before Theorem \ref{theo:unitarity_vosa_extensions}.
	With an abuse of notation, we identify $B$ with its corresponding object in $\mathcal{C}$. Then one can define the object $(B\boxtimes M, m_{B\boxtimes M})$ in $\operatorname{Mod}_\mathcal{C}(B)$ by defining the multiplication $m_{B\boxtimes M}:=(m\boxtimes \one_M)a_{B,B, M}$, where $a$ is the associator in $\Rep(V)$.
	Now, by \cite[Theorem 1.6 point 2.]{KO02}, cf.\ also \cite[Lemma 2.61]{CKM24}, we have that 
	\begin{equation}
		\Hom_{\operatorname{Mod}_\mathcal{C}(B)}(B\boxtimes M, U_\parone)
		=\Hom_\mathcal{C}(M, U_\parone)
	\end{equation}
	where the right hand side is non-zero as $M\subseteq U_\parone$. Therefore as $U_\parone$ is a $\Z_2$-simple current in $\Repu(U_\parzero)\equiv\operatorname{Mod}_\mathcal{C}^0(B)$, see Remark \ref{rem:niceness_even_odd_parts}, and thus in $\operatorname{Mod}_\mathcal{C}(B)$, the left hand side must be equivalent to $\C$ as claimed.
	Finally, if $\Rep(V)$ is pointed, then any simple object is a simple current and thus any simple CFT type VOSA extension of $V$ must be a $G$-graded simple current extension for some finite abelian group $G$.
\end{proof}

\subsection{VOSAs and graded-local conformal nets}
\label{subsec:preliminaries_VOSAs_gl_conformal_nets}

Definitions and properties of \textit{graded-local conformal nets} and their representation theory can be found in \cite[Sections 1--4]{CKL08}, see also \cite[Section 2]{CHKLX15}, \cite[Section 2]{CHL15} and \cite[Section 2]{CGH}.
We briefly recall from \cite[Section 4]{CGH}, cf.\ also \cite[Chapter 6]{CKLW18}, how to define a graded-local conformal net from a unitary VOSA, see Section \ref{subsec:preliminaries_unitary_VOSAs} for notations and results.

\begin{defin}  \label{defin:energy_bounds}
	Let $V$ be a unitary VOSA. A vector $a\in V$, or equivalently its corresponding vertex operator $Y(a,z)$, is said to satisfy the \textit{$k$-th order (polynomial) energy bounds} for some $k\in\R_{\geq0}$, if there exists $M, s\in \R_{\geq0}$ such that 
	\begin{equation} \label{eq:k-th_energy_bounds}
		\norm{a_nb}\leq M(1+\abs{n})^s\norm{(1_V+L_0)^kb}
		\qquad\forall n\in\half\Z \,\,\,\forall b\in V 
	\end{equation} 
	where $\norm{\cdot}$ is the norm on $V$ induced by the invariant scalar product $\scalar$. Furthermore, $a$ is said to satisfy \textit{linear energy bounds} if $k=1$; whereas it is simply said to satisfy \textit{energy bounds} whenever $k$ is not specified. Accordingly, $V$ is said to be \textit{energy-bounded} if all $a\in V$ satisfy energy bounds.
\end{defin} 

Let $\chi:S^1\to \C$ be the function defined by $\chi(z):=e^{i\frac{x}{2}}$ for the unique $x\in(-\pi,\pi]$ such that $z=e^{ix}$. Set $C_\parzero^\infty(S^1):=C^\infty(S^1)$, that is the vector space of \textit{infinite differentiable complex-valued functions} on the \textit{unit circle} $S^1$, and $C_\parone^\infty(S^1):=\{\chi h\mid h\in C^\infty(S^1)\}$. Both of these vector spaces are naturally equipped with a Fréchet topology as we define
\begin{equation}
	f'(z):=\eval{\frac{\mathrm{d}f(e^{ix})}{\mathrm{d}x}}_{e^{ix}=z}
	\qquad\mbox{and}\qquad
	g'(z):=\eval{\frac{\mathrm{d}g(e^{ix})}{\mathrm{d}x}}_{e^{ix}=z}
	=\chi(z)\left(\frac{i}{2}h(z)+h'(z)\right)
\end{equation}
for all $f\in C^\infty(S^1)$ and all $g=\chi h\in C^\infty_\parone(S^1)$.

Let $V$ be an energy-bounded unitary VOSA. Then it is not difficult to show that for all homogeneous $a\in V_\parzero$, all homogeneous $b\in V_\parone$, all $f\in C^\infty(S^1)$ and all $g=\chi h\in C^\infty_\parone(S^1)$, the operators defined by 
\begin{equation}
	\begin{split}
		Y_0(a,f)c
			&:=
		\sum_{n\in\Z} \widehat{f}_na_nc
		\quad\forall c\in V
		\,,\qquad
		\widehat{f}_n:=
		\frac{1}{2\pi}\int_{-\pi}^{\pi}f(e^{ix})e^{-inx}\mathrm{d}x  \\
		Y_0(b,g)c
			&:=
		\sum_{n\in\Z-\half} \widehat{g}_na_nc
		\quad\forall c\in V
		\,,\qquad
		\widehat{g}_n:=
		\frac{1}{2\pi}\int_{-\pi}^{\pi}h(e^{ix})e^{-i(n-\half)x}\mathrm{d}x 
	\end{split}
\end{equation}
are well-defined and closable on the Hilbert space completion $\mathcal{H}$ of $V$ with respect to $\scalar$. Their respective closures $Y(a,f)$ and $Y(b,g)$ are operator-valued distributions such that $Y(\overline{a},\overline{f})\subseteq Y(a,f)^*$ and $Y(\overline{b},\overline{g})\subseteq Y(b,g)^*$, where $\overline{g}(z):=\chi(z)\overline{zh(z)}$ for all $z\in S^1$. In particular, if $a$ (resp.\ $b$) is Hermitian and $f$ (resp.\ $g$) is real-valued, then $Y(a,f)$ (resp.\ $Y(b,g)$) is self-adjoint. Moreover, they have a common invariant domain, that is $\mathcal{H}^\infty$, which is the set of \textit{smooth vectors} for the (closure of the) operator $L_0$ on $\mathcal{H}$.

\begin{defin}
	Let $a\in V$ be any homogeneous vector. Then $Y(a,f)$ with $f\in C_{p(a)}^\infty(S^1)$ is called a \textit{smeared vertex operator}.
\end{defin}

Therefore, for an energy-bounded unitary VOSA $V$, we define
\begin{equation}
	\A_{(V,\scalar)}(I):=W^*\left(\left\{
	Y(a,f)\mid a\in V_n \,,\,\, n\in \half\Zpluseq \,,\,\, f\in C^\infty_{p(a)}(S^1)\,,\,\, \mathrm{supp}f\subset I
	\right\}\right)
	\,\,\forall I\in \J
\end{equation}
where $\J$ is the set of \textit{intervals} of $S^1$ and $W^*(S)$ denotes the \textit{von Neumann algebra generated} by the set $S$ of unbounded operators, see e.g.\ \cite[Section 5.2.7]{Ped89} and \cite[Appendix B.1]{Gui19I} for details.
It can be proved that if $V$ is a simple energy-bounded unitary VOSA, then $\A_{(V,\scalar)}$ is \textit{irreducible} and $Y(\nu,f)$ with real-valued $f\in C^\infty(S^1)$ integrates to a positive-energy strongly-continuous projective unitary representation $U$ on $\mathcal{H}$ of the infinite cover $\Diff^+(S^1)^{(\infty)}$ of the \textit{orientation preserving diffeomorphism group} $\Diff^+(S^1)$ of $S^1$, making $\A_{(V,\scalar)}$ \textit{M{\"o}bius covariant}. 
Actually, $U(r^{(\infty)}(4\pi))=1_\mathcal{H}$, where $\R\ni t\mapsto r^{(\infty)}(t)$ is the infinite cover of the one-parameter subgroup of \textit{rotations}, so that $U$ factors through a representation of the double cover $\Diff^+(S^1)^{(2)}$ of $\Diff^+(S^1)$. Moreover, $U(r^{(\infty)}(2\pi))=\Gamma$, that is the extension of $\Gamma_V$ to $\mathcal{H}$.
Nevertheless, the \textit{graded locality} is not automatically assured, see \cite[Section 10]{Nel59}, see also \cite[Section VIII.5]{RS80}. Therefore:

\begin{defin}
	Let $V$ be a unitary VOSA with invariant scalar product $\scalar$. Then $V$ is said to be \textit{strongly graded-local} if it is energy-bounded and $\A_{(V,\scalar)}$ satisfies the graded locality.
\end{defin}

\begin{theo}[{\cite[Theorem 4.29]{CGH}, \cite[Theorem 6.8]{CKLW18}}]
	For any simple strongly graded-local unitary VOSA $V$ with invariant scalar product $\scalar$, $\A_{(V,\scalar)}$ is diffeomorphism covariant and thus it is an irreducible graded-local conformal net. If $\curlyscalar$ is another unitary structure on $V$, then $(V,\curlyscalar)$ is strongly graded-local and $\A_{(V,\curlyscalar)}$ is (unitarily) isomorphic to $\A_{(V,\scalar)}$.
\end{theo}

\begin{defin}
	For any simple strongly graded-local unitary VOSA $V$, we denote by $\A_V$, the unique, up to isomorphism, irreducible graded-local conformal net associated to it.
\end{defin}

\begin{rem}   \label{rem:properties_strongly_graded-local_VOSAs}
	Recall that if $V$ is a simple strongly graded-local unitary VOSA, then: $\Aut_{\scalar}(V)$ is equivalent to the \textit{automorphism group} $\Aut(\A_V)$ of $\A_V$ by \cite[Theorem 4.32]{CGH}; unitary subalgebras of $V$ are strongly graded-local and they are in one-to-one correspondence with the \textit{covariant subnets} (see e.g.\ \cite[Section 2.3]{CGH} and references therein) of $\A_V$ by \cite[Theorem 6.1]{CGH}; if $G$ is a closed subgroup of $\Aut_{\scalar}(V)\equiv\Aut(\A_V)$, then the \textit{fixed-point subnet} $\A_V^G$ is (isomorphic to) $\A_{V^G}$ by \cite[Proposition 6.2]{CGH}; in particular, as the extension $\Gamma$ of the parity operator $\Gamma_V$ is the \textit{unitary grading} of $\A_V$, the \textit{Bose subnet} $(\A_V)_\parzero:=\A_{V}^{\langle 1_\mathcal{H}, \Gamma\rangle}$ of $\A_V$ (sometimes denoted by $\A_V^\Gamma$) is equal to $\A_{V_\parzero}$; if $W$ is a unitary subalgebra of $V$, then the \textit{coset subnet} $\A_W^c$ is (isomorphic to) $\A_{\Com_V(W)}$ by \cite[Proposition 6.3]{CGH}; if $U$ is any simple strongly graded-local unitary VOSA, then the \textit{graded tensor product} $\A_V\hat{\otimes} \A_U$ is (isomorphic to) $\A_{V\hat{\otimes}U}$ by \cite[Corollary 6.6]{CGH} (we will write $\A_V\otimes \A_U$ whenever $V$ or $U$ is a VOA, see p.\ \pageref{eq:defin_graded_tensor_product_vertex_ops}).
\end{rem}

We now recall some facts about the representation theory of (local) conformal nets. The \textit{(locally normal) representations}, also called \textit{DHR} representations, of an irreducible conformal net $\A$ give rise to a braided $C^*$-tensor category, denoted by $\Rep(\A)$, see \cite{FRS89, FRS92}, based on DHR superselection theory \cite{DHR71, DHR74, Haa96}. An equivalent approach in terms of Connes' fusion product is given in \cite{Was95, Was98}, see also \cite[Section 6]{Gui21} and \cite{Guib} for a recent review of the topic.
It is possible to give a notion of \textit{index} for representations of a conformal net $\A$, see e.g.\ \cite[Section 2.1]{GL96} and \cite[Section 2.2]{Car04}, so that the \textit{finite} index ones constitute a full subcategory $\Repf(\A)$ of $\Rep(\A)$.
The operator algebraic counterpart of strongly rational VOAs are morally represented by \textit{completely rational} conformal nets, see \cite[Definition 8]{KLM01}, that is those irreducible conformal nets satisfying \textit{strong additivity}, the \textit{split property} and having finite \textit{$\mu$-index}. We do not go into details of these properties, but we limit ourself to recall that by \cite[Theorem 4.9]{LX04} and \cite[Theorem 5.4]{MTW18}, a conformal net $\A$ is completely rational if and only if $\A$ has finitely many inequivalent irreducible representations and all of them have finite index or equivalently by \cite[Theorem 5.3]{LX04}, if and only if $\A$ has finite $\mu$-index. 
Therefore, let $\A$ be a completely rational conformal net, then: representations in $\Rep(\A)$ are (possibly infinite) direct sums of irreducible ones; whereas the representations in $\Repf(\A)$ are given by just finite direct sums; by \cite[Corollary 37 and Corollary 39]{KLM01}, $\Repf(\A)$ turns out to be a unitary modular tensor category.

It is possible to give a notion of complete rationality for irreducible graded-local conformal nets easily adapting the one in \cite[Definition 8]{KLM01}. 
Let $\A$ be an irreducible graded-local conformal net.
By \cite[Proposition 36]{Lon03}, which is based in part on a graded-local version of \cite[Lemma 22 and Lemma 23]{Lon03}, we have that $\A$ is completely rational if and only if its Bose subnet $\A_\parzero$ is
completely rational.
Moreover, similarly to the local case, we can consider \textit{(locally normal) representations}, also called \textit{DHR} or \textit{Neveu-Schawrz} representations, for the graded-local conformal net $\A$, see  \cite[Sections 2.4--2.5 and Section 4.3]{CKL08}, see also \cite[Section 2]{CHKLX15}, \cite[Section 2]{CHL15}.
Then we give the following definition, extending the usual one for the local case:

\begin{defin}
	A graded-local conformal net $\A$ is said to be \textit{holomorphic} if it is completely rational (or equivalently if its Bose subnet $\A_\parzero$ is) and it has only one irreducible (Neveu-Schawrz) representation, which in turn is given by $\A$ itself. 
\end{defin}

To link the representation theories of completely unitary VOAs and of completely rational conformal nets via strong locality, it is necessary to introduce three fundamental notions for intertwining operators of VOAs: \textit{energy boundedness}, \textit{strong intertwining property} and \textit{strong braiding}. These notions will be crucial to prove the strong graded-locality of some nice holomorphic VOSAs in Section \ref{subsec:gl_conformal_nets}. For this reason, we give here below a brief review, referring to \cite[Section 2]{Gui} and \cite[Section 4]{Gui21} for details.

Let $V$ be a VOA, $\Y$ be an intertwining operator of type $\binom{K}{M \,\,\, N}$ and $w\in M$ be any homogeneous vector. 
Then by definition, $\Y(w, z)$ is a formal series $\sum_{n\in \R}\Y(w)_nz^{-n-1}$, where every coefficient $\Y(w)_n$ is a linear map from $N$ to $K$.  
Suppose that $V$ and $\Y$ are unitary, then the definitions of \textit{$k$-th order/linear energy bounds} or simply \textit{energy bounds} for $\Y(w, z)$ will follow the ones given in Definition \ref{defin:energy_bounds}, see \cite[Section 4.4]{Gui21} and \cite[Section 2.1]{Gui} for details, cf.\ \cite[Section 3.1]{Gui19I}. Furthermore, we have:

\begin{defin}
	Let $V$ be a unitary VOA and $\Y$ be a unitary intertwining operator of type $\binom{K}{M \,\,\, N}$. Then $\Y$ is said to be \textit{energy-bounded} if for all homogeneous $w\in M$, $\Y(w,z)$ is energy-bounded. Any unitary $V$-module $M$ is said to be \textit{energy-bounded} if $Y^M$ is energy-bounded as unitary intertwining operator of type $\binom{M}{V \,\,\, M}$. Accordingly, if $V$ is a simple strongly unitary VOA, then it is said to be \textit{strongly energy-bounded} if every irreducible unitary $V$-module is energy-bounded. Furthermore, $V$ is said to be \textit{completely energy-bounded} if every irreducible unitary intertwining operator of $V$ is energy-bounded.
\end{defin} 

This allows us to define a smeared version of intertwining operators which we recall briefly, see \cite[Section 4.4]{Gui21} or \cite[Section 2.1]{Gui}, cf.\ \cite[Section 3.2]{Gui19I}.

First, see \cite[Section 3.1]{Gui21} or \cite[Section 1.1]{Gui} for details, we call an \textit{arg-function} defined on $I\in \J$ any continuous function $\arg_I:I\to \R$ such that for all $z=e^{ix}\in I$, $\arg_I(z)-x\in 2\pi\Z$. Accordingly, we set $\widetilde{I}:=(I,\arg_I)$, called an \textit{arg-valued interval}, which can be thought as one of the branch of $I$ in the universal cover of $S^1$. We will call $\widetilde{\J}$ the set of arg-valued intervals. We say that $\widetilde{I}$ and $\widetilde{J}$ in $\widetilde{\J}$ are \textit{disjoint} if $I$ and $J$ are. Moreover, we write $\widetilde{I}\subseteq\widetilde{J}$ if $I\subseteq J$ and $\arg_J\restriction_I=\arg_I$. We also say that $\widetilde{I}$ is \textit{anticlockwise} to $\widetilde{J}$ if they are disjoint and $\arg_J(\zeta)<\arg_I(z)<\arg_J(\zeta)+2\pi$ for all $\zeta\in J$ and all $z\in I$; of course, we also say that $\widetilde{J}$ is \textit{clockwise} to $\widetilde{I}$.
For any $\widetilde{I}=(I,\arg_I)\in\widetilde{\J}$ and any $f\in C^\infty(S^1)$ with $\mathrm{supp}f\subset I$, we call $\widetilde{f}:=(f,\arg_I)$ an \textit{arg-valued function}. Accordingly, $C_c^\infty(\widetilde{I})$ is the set of these arg-valued functions. 
Of course, for any inclusion $\widetilde{I}\subseteq\widetilde{J}$ of arg-valued intervals in $\widetilde{\J}$, $C_c^\infty(\widetilde{I})$ turns out to be a natural subspace of $C_c^\infty(\widetilde{J})$. 

For any unitary VOA module $M$, we denote by $\mathcal{H}_M$ its Hilbert space completion with respect to its invariant scalar product $\scalar_M$. Accordingly, if $M$ is energy-bounded, then we denote by $\mathcal{H}_M^\infty$ the subspace of $\mathcal{H}_M$ of smooth vectors for the operator $L_0^M$. 
We suppose that $V$ is a strongly unitary strongly rational VOA, so that every $V$-module is unitary and semisimple, that is it can be written into a finite sum of irreducible $V$-modules.
Let $\Y$ be a unitary intertwining operator of type $\binom{K}{M \,\,\, N}$ such that $\Y(w, z)=\sum_{n\in \R}\Y(w)_nz^{-n-1}$ is energy-bounded for some homogeneous $w\in M$ and let $\widetilde{f}\in C_c^\infty(\widetilde{I})$ for any $\widetilde{I}\in\widetilde{\J}$. 
Note that if $\Y$ is irreducible, then thanks to the translation property of intertwining operators, $\Y(w)_n\not=0$ only for a countable number of $n\in\R$, see \cite[Section 1.3]{Gui19I}.
The same is also true in the non-irreducible case thanks to the semisimplicity of $V$-modules.
Then we define the \textit{smeared intertwining operator} $\Y(w, \widetilde{f})$ as the sesquilinear form on $N\times K$ to $\C$ satisfying
\begin{equation}  
	(u|\Y(w, \widetilde{f})v)_K=\int_{\arg_I(I)}
	(u|\Y(w,e^{ix})v)_K f(e^{ix})\frac{e^{ix}\mathrm{d}x}{2\pi}
	\qquad\forall v\in N \,\,\,\forall u\in K
\end{equation}
where $\scalar_K$ is the invariant scalar product on $K$ (which is linear in the second variable in our setting).
We can actually regard $\Y(w, \widetilde{f})$ as the preclosed operator from $\mathcal{H}_N$ to $\mathcal{H}_K$ with domain $N$, defined by the formula
\begin{equation}   \label{eq:smeared_intertwining_op_series}
	\sum_{n\in\R}\Y(w)_n \widehat{\widetilde{f}}_n
	\,,\qquad
	\widehat{\widetilde{f}}_n:=
	\frac{1}{2\pi}\int_{\arg_I(I)}
	f(e^{ix})e^{-inx}\mathrm{d}x 
\end{equation}
mapping $N$ into $\mathcal{H}_K^\infty$, see \cite[Section 3.2]{Gui19I} for details.
Moreover, its closure maps $\mathcal{H}_N^\infty$ into $\mathcal{H}_K^\infty$. The symbol $\Y(w, \widetilde{f})$ will be usually used to denote the restriction to $\mathcal{H}_N^\infty$ of its closure as unbounded operator.

\begin{rem}  \label{rem:smeared_intertwining_and_vertex_ops}
	Let $\Y(w, \widetilde{f})$ be a smeared intertwining operator as in \eqref{eq:smeared_intertwining_op_series}. 
	If we choose $\Y=Y^M$ for some energy-bounded $V$-module $M$, then $n$ in \eqref{eq:smeared_intertwining_op_series} is an integer and $\widehat{\widetilde{f}}_n=\widehat{f}_n$ for all $n\in\Z$, see \cite[Remark 3.11]{Gui19I}. Moreover, we can extend \eqref{eq:smeared_intertwining_op_series} to all $f\in C^\infty(S^1)$ without any restriction on the supports.
	Now, suppose that $V$ is a simple energy-bounded unitary VOSA with $V_\parone\not=\{0\}$. Then the vertex operator $Y$ of $V$ can be naturally restricted to an irreducible energy-bounded unitary intertwining operator, say $Y_{\parone,\parone}$, of type $\binom{V_\parzero}{V_\parone \,\,\, V_\parone}$. 
	Therefore, for any homogeneous $b\in V_\parone$ and any $g=\chi h\in C^\infty_\parone(S^1)$ with $\mathrm{supp}g\subset I$, it is not difficult to see that
	$Y(b,g)$ coincides with $Y_{\parone,\parone}(b,\widetilde{f})$ on $V_\parone$, where $\widetilde{f}$ is the arg-valued function given by $f(z):=g(z)z^{d_a-1}:=h(z)z^{d_a-\half}$ and $\mathrm{arg}_I(z):=x$ where $z=e^{ix}$ for a unique $x\in(-\pi,\pi]$ for all $z\in S^1$. In particular, the restriction on the support of $g$ can be removed and $Y_{\parone,\parone}(b, \widetilde{f})$ can be defined for all $g\in C^\infty_\parone(S^1)$, still coinciding with $Y(b,g)$ on $V_\parone$. 
\end{rem}
                                                                              
\begin{defin}[{\cite[Definition 1.2.6]{Gui}, cf.\ \cite[Definition 4.6]{Gui21}}]
	Let $\mathcal{P}_0$, $\mathcal{Q}_0$, $\mathcal{R}_0$ and $\mathcal{S}_0$ be pre-Hilbert spaces with completions $\mathcal{P}$, $\mathcal{Q}$, $\mathcal{R}$ and $\mathcal{S}$ respectively. 
	Let $A:\mathcal{P}\to \mathcal{R}$, $B:\mathcal{Q}\to\mathcal{S}$, $C:\mathcal{P}\to \mathcal{Q}$ and $D:\mathcal{R}\to \mathcal{S}$ be preclosed operators whose domains are subspaces of $\mathcal{P}_0$, $\mathcal{Q}_0$, $\mathcal{R}_0$ and $\mathcal{S}_0$ respectively, and whose ranges are inside $\mathcal{P}_0$, $\mathcal{Q}_0$, $\mathcal{R}_0$ and $\mathcal{S}_0$ respectively. The following diagram of preclosed operators
	\begin{equation}
		\xymatrixcolsep{2cm}\xymatrixrowsep{2cm}\xymatrix{ \mathcal{P}_0 \ar[r]^{C} \ar[d]_{A} & \mathcal{Q}_0 \ar[d]_{B}\\		
			\mathcal{R}_0 \ar[r]^{D} & \mathcal{S}_0	} 
	\end{equation}
	is said to \textit{commute strongly} if the closures of the following preclosed operators
	\begin{equation}
		\begin{split}
			R(\xi\oplus\eta\oplus\chi\oplus\zeta)
				&:=
			0\oplus 0\oplus A\xi\oplus B\eta
			\qquad\forall 
			\xi\in\mathcal{D}(A)\,\,\,\forall \eta\in \mathcal{D}(B)
			\,\,\,\forall \chi\in \mathcal{R}\,\,\,\forall \zeta\in\mathcal{S}  \\
			 S(\xi\oplus\eta\oplus\chi\oplus\zeta)
			 &:=
			 0\oplus C\xi\oplus 0\oplus D\chi
			 \qquad\forall 
			 \xi\in\mathcal{D}(C)\,\,\,\forall \eta\in \mathcal{Q}
			 \,\,\,\forall \chi\in \mathcal{D}(D)\,\,\,\forall \zeta\in\mathcal{S}  
		\end{split}
	\end{equation}
	with corresponding domains, in the Hilbert space $\mathcal{P}\oplus\mathcal{Q}\oplus\mathcal{R}\oplus\mathcal{S}$,
	\begin{equation}
			\mathcal{D}(R):=\mathcal{D}(A)\oplus\mathcal{D}(B)\oplus\mathcal{R}\oplus\mathcal{S} \qquad\mbox{and}\qquad
			\mathcal{D}(S):=\mathcal{D}(C)\oplus\mathcal{Q}\oplus\mathcal{D}(D)\oplus\mathcal{S}
	\end{equation}
	 commute strongly, that is the von Neumann algebras generated by them commute.
\end{defin}

Let $V$ be a strongly rational unitary VOA, then for a \textit{full rigid monoidal subcategory} $\mathcal{C}$ of $\Repu(V)$ we mean a class of objects of $\Repu(V)$ such that: it includes the identity object $V$; if $M\in\mathcal{C}$, then $M'$ is equivalent to an object in $\mathcal{C}$; if $M\in\mathcal{C}$, then any suboject of $M$ is equivalent to an object in $\mathcal{C}$; if $M,N\in\mathcal{C}$, then $M\boxtimes N$ is equivalent to an object in $\mathcal{C}$. 
We say that a set $\mathcal{F}$ of objects of $\mathcal{C}$ \textit{generates} $\mathcal{C}$ if any irreducible object of $\mathcal{C}$ is equivalent to a subobject of a tensor product of elements in $\mathcal{F}$.
Therefore, following the beginning of \cite[Section 2.3]{Gui} and \cite[Definition 4.10]{Gui21}, we give:

\begin{defin}   \label{defin:strong_intertwining}
	Let $V$ be a strongly rational unitary VOA. 
	\begin{itemize}
		\item[(i)] 
		Suppose that $V$ is strongly energy-bounded. 
		Let $\Y$ be an energy-bounded unitary intertwining operator of type $\binom{K}{M\,\,\, N}$ and let $w\in M$ be a homogeneous vector. 
		Then $\Y(w,z)$ is said to satisfy the \textit{strong intertwining property} if for all homogeneous $v\in V$, all $\widetilde{I}\in\widetilde{\J}$, all $J\in \J$ disjoint from $I$, all $\widetilde{f}\in C^\infty_c(\widetilde{I})$ and all $g\in C^\infty(S^1)$ with $\mathrm{supp}g\subset I$, the following diagram of precolsed operators commutes strongly
		\begin{equation} \label{eq:strong_intertwining_diagram}
			\xymatrixcolsep{2cm}\xymatrixrowsep{2cm}\xymatrix{ \mathcal{H}_N^\infty \ar[r]^{Y^N(v,g)} \ar[d]_{\Y(w,\widetilde{f})} & \mathcal{H}_N^\infty \ar[d]_{\Y(w,\widetilde{f})}\\		
				\mathcal{H}_K^\infty \ar[r]^{Y^K(v,g)} & \mathcal{H}_K^\infty	}
		\end{equation}
		Accordingly, $\Y$ is said to satisfy the \textit{strong intertwining property} if $\Y(w,z)$ does for all homogeneous vector $w\in M$.
		
		\item[(ii)] 
		Let $\mathcal{C}$ be a full rigid monoidal subcategory of $\Repu(V)$. Suppose that $V$ is strongly local and that any $V$-module in $\mathcal{C}$ is energy-bounded. Let $w$ be a homogeneous vector in a $V$-module $M\in \mathcal{C}$. Then the action of $w$ on $\mathcal{C}$ is said to satisfy the \textit{strong intertwining property} if for all $N, K\in\mathcal{C}$ and all $\Y\in\mathcal{V}\binom{K}{M\,\,\, N}$, $\Y(w,z)$ is energy-bounded and it satisfies the strong intertwining property.
	\end{itemize} 
\end{defin}

\begin{rem} \label{rem:strong_locality_from_strong_intertwining}
	Note that if we choose $\Y$ as the vertex operator $Y$ of $V$ in (i) of Definition \ref{defin:strong_intertwining}, so that $M=N=K=V$, then stating that $Y$ satisfies the strong intertwining property, it is equivalent to state that $V$ is strongly local. 
\end{rem}

In \cite[Section 4.2]{Gui21}, see also \cite[Section 2.5]{Gui}, the author introduce a pair of operators $\mathcal{L}_M$ and $\mathcal{R}_M$ for any semisimple $V$-module $M$. For the sake of simplicity, let suppose that $V$ is strongly rational. 
Let $N$ be any $V$-module and recall the definition of $M\boxtimes N:=
\bigoplus_{k\in\mathcal{E}}\mathcal{V}\binom{M_k}{M \,\,\, N}^*\otimes M_k$ given in \eqref{eq:defin_cat_tensor_product}. 
Then the operator $\mathcal{L}_M$ acts on $N$ as an intertwining operator of type $\binom{M\boxtimes N}{M \,\,\, N}$ defined as follows: 
for all $v\in M$, all $w\in N$, all $k\in \mathcal{E}$, all $u^k\in M_k'$ and all $\Y\in\mathcal{V}\binom{M_k}{M\,\,\, N}$, it must satisfies
\begin{equation}   \label{eq:defin_left_operator}
	\langle
	\mathcal{L}_M(v,z)w, \Y\otimes u^k
	\rangle
	=
	\langle
	\Y(v,z)w,  u^k
	\rangle
\end{equation}
where $\pairing$ in the left hand side is the natural pairing between $\mathcal{V}\binom{M_k}{M \,\,\, N}^*\otimes M_k$ and $\mathcal{V}\binom{M_k}{M \,\,\, N}\otimes M_k'$ for all $k\in \mathcal{E}$, whereas in the right hand side it is the one between $M_k$ and $M_k'$ for all $k\in \mathcal{E}$.
Similarly, with notations as here above, $\mathcal{R}_M$ acts as an intertwining operator of type $\binom{N\boxtimes M}{M \,\,\, N}$ defined by 
\begin{equation} \label{eq:defin_right_operator}
	\mathcal{R}_M(v,z)w:=b_{M,N}\mathcal{L}_M(v,z)w
\end{equation}
where $b_{M,N}$ is the braiding in $\Rep(V)$.
As at the end of \cite[Section 2.5]{Gui} and in \cite[Definition 4.13]{Gui21}, we give the following:

\begin{defin}  \label{defin:strong_braiding}
		Let $V$ be a strongly rational unitary VOA. 
		\begin{itemize}
			\item[(i)]
			Suppose that $V$ is completely unitary and completely energy-bounded. The unitary intertwining operators of $V$ are said to satisfy the \textit{strong braiding} if for all $M,N,K\in\Repu(V)$, for all homogeneous $w\in M$ and all homogeneous $u\in N$, for all $\widetilde{I},\widetilde{J}\in\widetilde{\J}$ such that $\widetilde{I}$ is anticlockwise to $\widetilde{J}$, all $\widetilde{f}\in C^\infty_c(\widetilde{I})$ and all $\widetilde{g}\in C^\infty_c(\widetilde{J})$, the following diagram of precolsed operators commutes strongly
			\begin{equation}  \label{eq:strong_braiding_diagram}
				\xymatrixcolsep{2cm}\xymatrixrowsep{2cm}\xymatrix{ \mathcal{H}_K^\infty \ar[r]^{\mathcal{R}_N(u,\widetilde{g})} \ar[d]_{\mathcal{L}_M(w,\widetilde{f})} & \mathcal{H}_{K\boxtimes N}^\infty \ar[d]_{\mathcal{L}_M(w,\widetilde{f})}\\		
					\mathcal{H}_{M\boxtimes K}^\infty \ar[r]^{\mathcal{R}_N(u,\widetilde{g})} & \mathcal{H}_{M\boxtimes K\boxtimes N}^\infty	} 
			\end{equation}
			
			\item[(ii)] 
			Let $\mathcal{C}$ be a full rigid monoidal subcategory of $\Repu(V)$. Suppose that $V$ is strongly local and that any $V$-module in $\mathcal{C}$ is energy-bounded. Let $M,N\in\mathcal{C}$ and choose $w\in M$ and $u\in N$ be homogeneous vectors. Then the actions of $w$ and of $u$ on $\mathcal{C}$ are said to satisfies the \textit{strong braiding property} if for all $K\in \mathcal{C}$, all $\widetilde{I},\widetilde{J}\in\widetilde{\J}$ such that $\widehat{I}$ is anticlockwise to $\widetilde{J}$, all $\widetilde{f}\in C^\infty_c(\widetilde{I})$ and all $\widetilde{g}\in C^\infty_c(\widetilde{J})$, the diagram \eqref{eq:strong_braiding_diagram} of precolsed operators commutes strongly.
		\end{itemize}
\end{defin}

\begin{rem}   \label{rem:strong_intertwining_from_strong_braiding}
	Suppose that the strong braiding property in (i) of Definition \ref{defin:strong_braiding} holds with $N=V$ and for all homogeneous $u\in N$. Then the strong intertwining property follows, see \cite[Remark 2.5.9]{Gui} and the discussion after \cite[Theorem 4.11]{Gui21}. In particular, $V$ turns out to be strongly local, see Remark \ref{rem:strong_locality_from_strong_intertwining}. 
\end{rem}

Recalling Lemma \ref{lem:simple_current_extensions}, we prove the following result which will be crucial in the proof of Theorem \ref{theo:strong_graded_locality_holo}.

\begin{lem}   \label{lem:strong_graded_locality_simple_current_ext}
	Let $V$ be a strongly local and completely unitary VOA. Let $U=\bigoplus_{\xi\in G} V^\xi$ be a $G$-graded simple current extension of $V$ for a finite abelian group $G$ and let $\mathcal{G}$ be the full rigid monoidal subcategory of $\Repu(V)$ generated by these simple currents. Set $\dot{G}:=G\backslash\{1_G\}$.
	Suppose that $U$ is an energy-bounded VOSA and that: for all $\xi\in \dot{G}$, there exists a non-zero quasi-primary vector $w^\xi\in V^\xi$ such that its action on $\mathcal{G}$ satisfies the strong intertwining property; for all $\xi,\eta\in \dot{G}$, the actions of $w^\xi\in V^\xi$ and $w^\eta\in V^\eta$ on $\mathcal{G}$ satisfy the strong braiding property.
	Then $U$ is strongly graded-local.
\end{lem}

\begin{proof}
	The idea of the proof is to generalize what is done in the proof of \cite[Theorem 7.19]{CGH} for $\Z_2$-graded simple current extensions to the case of generic $G$-graded simple current extensions.
	
	Let $\widehat{G}$ be the dual group of $G$, that is every $\phi$ is a continuous homomorphism from $G$ to the circle group $\T$.  
	Then $\widehat{G}$ has a natural action through automorphisms of $V$ given by $\phi(a):=\phi(\eta) a$ for all $a\in V^\eta$. By \cite{GN03, DZ08}, $V$ is finitely generated as it is strongly rational. Therefore, $U$ is finitely generated too and there exists a unitary structure $\curlyscalar$ on $U$ making $\widehat{G}$ a subset of $\Aut_{\curlyscalar}(U)$, see \cite[Proposition 3.23]{CGH}.
	This implies that the $V^\xi$'s are orthogonal to each other with respect to $\curlyscalar$.
	From now on, we fix on $U$ such unitary structure determined by $\curlyscalar$.
	Of course, the unitary structure on $U$ restricts to unitary structures of $V$-modules on the simple currents.
	Furthermore, the orthogonality of the $V^\xi$'s implies that the Hilbert space completion $\mathcal{H}_U$ of $U$ with respect to $\curlyscalar$ is equal to the direct sum of the Hilbert space completions $\mathcal{H}_\xi$ of $V^\xi$ for all $\xi\in G$ with respect to the corresponding restrictions of $\curlyscalar$. Moreover, we have that $\mathcal{H}_U^\infty=\bigoplus_{\xi\in G}\mathcal{H}_\xi^\infty$.
	Note that the complete unitarity of $V$ assures that $\Repu(V)$ is a unitary modular tensor category. Accordingly, the equivalence between $V^\xi\boxtimes V^\eta$ and $V^{\xi\eta}$ is an equivalence of unitary $V$-modules and thus we can safely identify them in what follows.
	
	For all $\xi\in \dot{G}$ and all $\eta\in G$, the restriction $Y_{\xi,\eta}(w^\xi,z)$ of the vertex operator $Y(w^\xi,z)$ to the $V$-module $V^\eta$ defines an irreducible unitary intertwining operator of $V$ of type $\binom{V^{\xi\eta}}{V^{\xi} \,\,\, V^{\eta}}$, shortly $\binom{\xi\eta}{\xi\,\,\,\eta}$. Moreover, every such intertwining operator is energy-bounded as $U$ is. 
	For all $I\in\J$, define the continuous function $\mathrm{arg}_I:I\to \R$ by $\mathrm{arg}_I(z)=x$ where $z=e^{ix}$ for a unique $x\in(-\pi,\pi]$. Then for all $w^\xi$ and all $f\in C_{p(w^\xi)}^\infty(S^1)$ with $\mathrm{supp}f\subset I$ for some $I\in\J$, we have that 
	\begin{equation}
		Y_{\xi,\eta}(w^\xi,\widetilde{f})=Y(w^\xi, f)\restriction_{\mathcal{H}_\eta^\infty}
		\qquad\forall \xi\in\dot{G}\,,\,\,\,\forall \eta\in G
	\end{equation}
	where $\widetilde{f}$ is the arg-valued function $(f\cdot z^{d_{w^\xi}-1},\mathrm{arg}_I)$, see Remark \ref{rem:smeared_intertwining_and_vertex_ops}. 
	For all $\xi\in \dot{G}$ and all $\eta\in G$, let $\mathcal{L}_{V^\xi}$ and $\mathcal{R}_{V^\xi}$ be the operators defined in \eqref{eq:defin_left_operator} and \eqref{eq:defin_right_operator} respectively. 
	Then there exists $a\in \C$ such that  $\mathcal{L}_{V^\xi}\restriction_{V^\eta}=a Y_{\xi,\eta}$.
	Let $w\in V^\eta$ be any homogeneous vector, then
	\begin{equation}
		\begin{split}
			\mathcal{R}_{V^\xi}(w^\xi,z)w
			&=
			b_{V^\xi,V^\eta}\mathcal{L}_{V^\xi}(w^\xi,z)w \\
			&=
			ab_{V^\xi,V^\eta}Y_{\xi,\eta}(w^\xi,z)w \\
			&=
			ae^{zL_{-1}}Y_{\xi,\eta}(w^\xi,-z)w \\
			&=
			a(-1)^{p(w^\xi)p(w)}Y_{\xi,\eta}(w^\xi,z)w \\
			&=
			(-i)^{p(w^\xi)}Z_V\mathcal{L}_{V^\xi}(w^\xi,z)Z_V^*w
		\end{split}
	\end{equation}
	where we have used the definition of the braiding, see eq.s \eqref{eq:braided_intertwining_operator} and \eqref{eq:defin_braiding}, for the third equality and the skew symmetry \cite[Section 4.2]{Kac01} of $U$ for the fourth one. 
	
	Fix $\xi\in\dot{G}$ and $\eta\in G$.
	By the strong braiding property, we have that for all $\zeta\in G$, the following diagram of preclosed operators 
	\begin{equation} \label{eq:strong_braiding_simple_currents}
		\xymatrixcolsep{4cm}\xymatrixrowsep{4cm}\xymatrix{ \mathcal{H}_\zeta^\infty \ar[r]^{(-i)^{p(w^\xi)}ZY_{\xi,\zeta}(w^\xi,\widetilde{g})Z^*} \ar[d]_{Y_{\eta,\zeta}(w^\eta,\widetilde{f})} & \mathcal{H}_{\zeta\xi}^\infty \ar[d]_{Y_{\eta,\zeta\xi}(w^\eta,\widetilde{f})}\\		
			\mathcal{H}_{\eta\zeta}^\infty \ar[r]^{(-i)^{p(w^\xi)}ZY_{\xi,\eta\zeta}(w^\xi,\widetilde{g})Z^*} & \mathcal{H}_{\eta\zeta\xi}^\infty }
	\end{equation}
	commutes strongly, that is the preclosed operators $R_\zeta$ and $S_\zeta$ on the Hilbert space $\mathcal{K}_\zeta:=\mathcal{H}_\zeta\oplus \mathcal{H}_{\zeta\xi}\oplus \mathcal{H}_{\eta\zeta}\oplus \mathcal{H}_{\eta\zeta\xi}$ defined by
	\begin{equation}
		\begin{split}
			R_\zeta(a\oplus b\oplus c\oplus d)
			&:=
			0\oplus 0\oplus  Y_{\eta,\zeta}(w^\eta,\widetilde{f})a\oplus Y_{\eta,\zeta\xi}(w^\eta,\widetilde{f})b \\
			S_\zeta	(a\oplus b\oplus c\oplus d)
			&:=
			0\oplus (-i)^{p(w^\xi)}ZY_{\xi,\zeta}(w^\xi,\widetilde{g})Z^*a \oplus 0\oplus (-i)^{p(w^\xi)}ZY_{\xi,\eta\zeta}(w^\xi,\widetilde{g})Z^*c
		\end{split}
	\end{equation} 
	with domains
	\begin{equation}
		\mathcal{D}(R_\zeta):=\mathcal{H}_\zeta^\infty\oplus \mathcal{H}_{\zeta\xi}^\infty\oplus \mathcal{H}_{\eta\zeta}\oplus \mathcal{H}_{\eta\zeta\xi}
		\qquad\mbox{and}\qquad
		\mathcal{D}(S_\zeta):=\mathcal{H}_\zeta^\infty\oplus \mathcal{H}_{\zeta\xi}\oplus \mathcal{H}_{\eta\zeta}^\infty\oplus \mathcal{H}_{\eta\zeta\xi}
	\end{equation}
	commute strongly.
	It is useful to rewrite these operators in the matrix form:
	\begin{equation}
		\begin{split}
			R_\zeta&:=\left(\begin{array}{cccc}
				0 & 0 & 0 & 0 \\
				0 & 0 & 0 & 0 \\
				Y_{\eta,\zeta}(w^\eta,\widetilde{f}) & 0 & 0 & 0 \\
				0 & Y_{\eta,\zeta\xi}(w^\eta,\widetilde{f}) & 0 & 0
			\end{array}
			\right)
			\\
			S_\zeta&:=(-i)^{p(w^\xi)}\left(\begin{array}{cccc}
				0 & 0 & 0 & 0 \\
				ZY_{\xi,\zeta}(w^\xi,\widetilde{g})Z^* & 0 & 0 & 0 \\
				0 & 0 & 0 & 0 \\
				0 & 0 & ZY_{\xi,\eta\zeta}(w^\xi,\widetilde{g})Z^* & 0
			\end{array}
			\right) \,.
		\end{split}
	\end{equation}
	Fix an ordering on the elements of $G$, so that $G=\{\zeta_1=1_G,\dots, \zeta_{\abs{G}}\}$. 
	Therefore, on the Hilbert space $\mathcal{K}:=\bigoplus_{j=1}^{\abs{G}}\mathcal{K}_{\zeta_j}$, consider the matrix $R$ and $S$ of preclosed operators, having on the diagonal the blocks of matrices $R_{\zeta_1},\dots, R_{\zeta_{\abs{G}}}$ and $S_{\zeta_1},\dots, S_{\zeta_{\abs{G}}}$ respectively and zero elsewhere.
	Using the polar decomposition as in \cite[Proposition B.5]{Gui19I} for the closures of the operators $R$ and $S$, it is easy to see that they commute strongly as their components do.
	Note that the Hilbert space $\mathcal{K}$ has exactly four copies of every $\mathcal{H}_\zeta$ for all $\zeta\in G$. Moreover, $R$ has exactly two copies of every smeared intertwining operator $Y_{\eta,\zeta}(w^\eta,\widetilde{f})$ for all $\zeta \in G$. Similarly, $S$ has exactly two copies of every $(-i)^{p(w^\xi)}ZY_{\xi,\zeta}(w^\xi,\widetilde{g})Z^*$ for all $\zeta \in G$.
	Consider the unitary operator $u$ permuting the copies of the Hilbert spaces in $\mathcal{K}$ in such a way that $\mathcal{K}$ is mapped into 
	\begin{equation}
		\bigoplus_{j=1}^{\abs{G}}\mathcal{H}_{\zeta_j}
		\oplus
		\bigoplus_{j=1}^{\abs{G}}\mathcal{H}_{(\zeta_j\xi^{-1})\xi}
		\oplus
		\bigoplus_{j=1}^{\abs{G}}\mathcal{H}_{\eta(\eta^{-1}\zeta_j)}
		\oplus
		\bigoplus_{j=1}^{\abs{G}}\mathcal{H}_{\eta(\eta^{-1}\zeta_j\xi^{-1})\xi}
		=\mathcal{H}_U\oplus\mathcal{H}_U\oplus\mathcal{H}_U\oplus\mathcal{H}_U \,.
	\end{equation}
	It is not difficult to see that the operators $uRu^*$ and $uSu^*$ send a generic element 
	\begin{equation}
		a\oplus b\oplus c\oplus d=
		\bigoplus_{j=1}^{\abs{G}}a_{\zeta_j}
		\oplus
		\bigoplus_{j=1}^{\abs{G}}b_{\zeta_j}
		\oplus
		\bigoplus_{j=1}^{\abs{G}}c_{\zeta_j}
		\oplus
		\bigoplus_{j=1}^{\abs{G}}d_{\zeta_j}
	\end{equation}
	of $\mathcal{H}_U\oplus\mathcal{H}_U\oplus\mathcal{H}_U\oplus\mathcal{H}_U$ into
	\begin{equation}
		\begin{split}
			&
			0\oplus 0\oplus
			\bigoplus_{j=1}^{\abs{G}}Y_{\eta, \eta^{-1}\zeta_j}(w^\eta, \widetilde{f})a_{\eta^{-1}\zeta_j} \oplus
			\bigoplus_{j=1}^{\abs{G}}
			Y_{\eta, \eta^{-1}\zeta_j}(w^\eta, \widetilde{f})b_{\eta^{-1}\zeta_j}
			\\
			&
			0\oplus 
			\bigoplus_{j=1}^{\abs{G}}(-i)^{p(w^\xi)}ZY_{\xi,\zeta_j\xi^{-1}}(w^\xi,\widetilde{g})Z^*a_{\zeta_j\xi^{-1}}
			\oplus 0 \oplus
			\bigoplus_{j=1}^{\abs{G}}(-i)^{p(w^\xi)}ZY_{\xi,\zeta_j\xi^{-1}}(w^\xi,\widetilde{g})Z^*c_{\zeta_j\xi^{-1}}
		\end{split}
	\end{equation}
	whenever $a,b$ and $c$ are in $\mathcal{H}_U^\infty=\bigoplus_{j=1}^{\abs{G}}\mathcal{H}_\xi^\infty$ respectively. 
	Then note that 
	\begin{equation}
		\begin{split}
			Y(w^\eta, f)a
			&=
			\bigoplus_{j=1}^{\abs{G}}Y_{\eta, \eta^{-1}\zeta_j}(w^\eta, \widetilde{f})a_{\eta^{-1}\zeta_j} \\
			Y(w^\eta, f)b
			&=
			\bigoplus_{j=1}^{\abs{G}}
			Y_{\eta, \eta^{-1}\zeta_j}(w^\eta, \widetilde{f})b_{\eta^{-1}\zeta_j} \\
			(-i)^{p(w^\xi)}ZY(w^\xi,g)Z^*a
			&=
			\bigoplus_{j=1}^{\abs{G}}(-i)^{p(w^\xi)}ZY_{\xi,\zeta_j\xi^{-1}}(w^\xi,\widetilde{g})Z^*a_{\zeta_j\xi^{-1}} \\
			(-i)^{p(w^\xi)}ZY(w^\xi,g)Z^*c
			&=
			\bigoplus_{j=1}^{\abs{G}}(-i)^{p(w^\xi)}ZY_{\xi,\zeta_j\xi^{-1}}(w^\xi,\widetilde{g})Z^*c_{\zeta_j\xi^{-1}} \,.
		\end{split}
	\end{equation}
	Rewriting everything in terms of matrices, we have found that the two preclosed operators 
	\begin{equation}
		\begin{split}
			uR u^*&:=\left(\begin{array}{cccc}
				0 & 0 & 0 & 0 \\
				0 & 0 & 0 & 0 \\
				Y(w^\eta,f)\restriction_{\mathcal{H}_U^\infty} & 0 & 0 & 0 \\
				0 & Y(w^\eta,f)\restriction_{\mathcal{H}_U^\infty} & 0 & 0
			\end{array}
			\right)
			\\
			uSu^*&:=(-i)^{p(w^\xi)}\left(\begin{array}{cccc}
				0 & 0 & 0 & 0 \\
				ZY(w^\xi,g)\restriction_{\mathcal{H}_U^\infty}Z^* & 0 & 0 & 0 \\
				0 & 0 & 0 & 0 \\
				0 & 0 & ZY(w^\xi,g)\restriction_{\mathcal{H}_U^\infty}Z^* & 0
			\end{array}
			\right) 
		\end{split}
	\end{equation}
	on the Hilbert space $\mathcal{H}_U\oplus\mathcal{H}_U\oplus\mathcal{H}_U\oplus\mathcal{H}_U$ commute strongly.
	Using again the polar decomposition as in \cite[Proposition B.5]{Gui19I} for the closures of the operators $uR u^*$ and $uSu^*$, we can conclude from the strong commutativity of $uRu^*$ and $uSu^*$ that for all $\xi,\eta\in \dot{G}$, the operators $Y(w^\eta,f)$ and $ZY(w^\xi,g)Z^*$ commute strongly. 
	
	The same argument applied by using the strong intertwining property in place of the strong braiding shows that for all homogeneous $v\in V$ and all $\xi\in \dot{G}$, the operators $Y(v,f)$ and $ZY(w^\xi,g)Z^*$ commute strongly (cf.\ Remark \ref{rem:strong_intertwining_from_strong_braiding}).
	To conclude recall that $V$ is strongly local by hypothesis. Moreover, the subset $V\cup \{w^\xi\mid \xi\in \dot{G}\}$ of $U$ generates $U$ itself, see \cite[Proposition 4.5.6]{LL04}, cf.\ also \cite[Lemma 6.1.1]{Li94}. Therefore, we can infer the strong graded locality of $U$ by adapting the proof of \cite[Theorem 6.8, see also Remark 6.9]{CGH}, cf.\ the proof of \cite[Theorem 5.5]{CGGH23}.
\end{proof}

To complete the exposition, we recall how the properties just introduced link the representation theories of completely unitary VOAs and completely rational conformal nets.
The following definition is due to \cite{CWX}, see also \cite[Section 4.5]{Gui21} or \cite[Section 2.3]{Gui}.
	
\begin{defin} \label{defin:strong-integrability}
		Let $V$ be a strongly rational and strongly local unitary VOA and let $M$ be a unitary $V$-module. Then $M$ is said to be \textit{strongly-integrable} if it is energy-bounded and there exists a (necessarily unique) representation $\pi_M$ in $\Rep(\A_V)$, acting on the Hilbert space $\mathcal{H}_M$, such that for all homogeneous $a\in V$, all $I\in\J$, all $f\in C^\infty(S^1)$ with $\mathrm{supp}\subset I$,
		$\pi_M(Y(a,f))$ (defined as at the end of \cite[Section B.1]{Gui19I} using the polar decomposition of $Y(a,f)$) is equal to the closure of $Y^M(a,f)$ on $\mathcal{H}_M$.
\end{defin}

Let $V$ be a strongly local, strongly unitary and strongly rational VOA and let $\mathcal{C}$ be a full rigid monoidal subcategory of $\Repu(V)$. It is proved in \cite[Theorem 4.11]{Gui21} that if there is a generating set $\mathcal{F}$ of irreducible objects for $\mathcal{C}$ such that for all $M\in\mathcal{F}$ there exists a non-zero homogeneous $w\in M$ whose action on $\mathcal{C}$ satisfies the strong intertwining property, then every $V$-module in $\mathcal{C}$ is strongly integrable. Furthermore, for all $M, N\in \mathcal{C}$, $M\boxtimes N$ is positive, so that $\mathcal{C}$ is a \textit{unitary ribbon fusion category}, see \cite[Theorem 4.22 and Remark 4.21]{Gui21}, see also \cite[Theorem 2.4.1]{Gui}, all based on the results in \cite{Gui19II}. In particular, if $\mathcal{C}$ is equivalent to $\Repu(V)$, then $V$ is completely unitary. 

In general, one can define the \textit{strong integrability $^*$-functor}
	\begin{equation}  \label{eq:defin_strong_integrability_functor}
		\mathfrak{F}:\mathcal{C}\to \Rep(\A_V) 
		\,,\quad M\mapsto (\pi_M,\mathcal{H}_M)
	\end{equation}
which is proved to be a fully faithful $^*$-functor, see \cite{CWX} or \cite[Theorem 4.3]{Gui19II}.

\begin{theo}[{\cite[Theorem II, Theorem 2.6.6, Corollary 2.6.7 and Theorem 2.6.8]{Gui}}]  \label{theo:strong_braiding_complete_energy_boundedness}
	Let $V$ be a strongly local, strongly energy-bounded, strongly unitary and strongly rational VOA. Assume that there exists a generating set $\mathcal{F}$ of $\Repu(V)$ such that for all $M\in\mathcal{F}$ there exists a non-zero quasi-primary vector $w\in M$ such that its action on $\Repu(V)$ satisfies the strong intertwining property (this is basically \cite[Condition B]{Gui}). Then:
	\begin{itemize}
		\item[(i)] $V$ is completely unitary and $\Repu(V)$ is a unitary modular tensor category;
		
		\item [(ii)] $\Repu(V)$ is equivalent to a braided $C^*$-tensor subcategory of $\Rep(\A_V)$ via the strong integrability $^*$-functor $\mathfrak{F}$;
		
		\item [(iii)] for any $M,N\in\Repu(V)$ such that unitary intertwining operators of type $\binom{\cdot}{M \,\,\,\cdot}$ and of type $\binom{\cdot}{N \,\,\,\cdot}$ are energy-bounded, the actions of all homogeneous vectors of $M$ and $N$ on $\Repu(V)$ satisfy the strong intertwining property and the strong braiding one;
		
		\item[(iv)] suppose that $V$ is a unitary subalgebra of a VOA $U$ satisfying the same conditions, then $\Com_U(V)$ also satisfies the same conditions if it is strongly rational; if $U$ is completely energy-bounded, then $\Com_U(V)$ is too and all its unitary intertwining operators satisfy the strong intertwining property and the strong braiding;
		
		\item[(v)]  if $W$ is a VOA satisfying the same conditions on $V$, then $V\otimes W$ does too; if $V$ and $W$ are completely energy-bounded, then $V\otimes W$ is too and  all their unitary intertwining operators satisfy the strong intertwining property and the strong braiding.
	\end{itemize}
\end{theo}

\begin{rem}  \label{rem:examples_strong_braiding_complete_energy_boundedness}
	Recall that lattice VOSAs, see Remark \ref{rem:lattice_vosas}, are strongly graded-local by \cite[Theorem 7.19]{CGH}. 
	Moreover, lattice VOAs satisfy the hypotheses of Theorem \ref{theo:strong_braiding_complete_energy_boundedness} and are completely energy-bounded by \cite[Theorem 5.8]{Gui21} and \cite[Theorem 2.7.9]{Gui}. Other VOAs satisfying those hypotheses and the complete energy boundedness are given in \cite[Section 2.7]{Gui}; some references about these models can be found in \cite[Table 1]{CGGH23}  (note that not all VOAs there satisfy the complete energy boundedness).
	Among them, there is the \textit{unitary discrete series of Virasoro VOAs with central charge $c<1$}, usually denoted by $L(c,0)$.
\end{rem}

\begin{rem}
	The theory developed above allows also to link VOSA extensions with graded-local conformal net ones, see \cite[Section 4]{CGGH23} and references therein.
\end{rem}

\section{Main results}

\subsection{Unitarity}
\label{subsec:unitarity}

To prove the unitarity of nice holomorphic VOSAs with central charge $c\leq 24$, we will rely on the tools developed for their classification given in \cite{HM}. 
Differently from \cite[Convention 2.2]{HM}, the odd parts of our VOSAs can be zero.
We also use the notation and the results presented in Section \ref{subsec:preliminaries_unitary_VOSAs}.

We denote by $F^l$ the graded tensor product of $l\in\Zpluseq$ copies of the \textit{real free fermion VOSA} $F$, see e.g.\ \cite[Section 3.6 and Prposition 4.10(b)]{Kac01} and \cite[Section 2.3]{AL17}. Therefore, one of the tools for the classification result is the \textit{Free-Fermion Splitting}:

\begin{prop}[{\cite[Section 2.6 and Section 3.2]{HM}}]   \label{prop:free-fermion-splitting}
	For all $c\in\half\Zpluseq$ and all $l\in\Zpluseq$, the map $V\mapsto V\hat{\otimes}F^l$ defines a bijection between isomorphism classes of nice holomorphic VOSAs with central charge $c$ and with weight-$\half$ subspace of dimension $k\in\Zpluseq$ and isomorphism classes of nice holomorphic VOSAs with central charge $c+\frac{l}{2}$ and with weight-$\half$ subspace of dimension $k+l$.
	In particular, every nice holomorphic VOSA $V$ with central charge $c\in\half\Zpluseq$ is isomorphic to a graded tensor product $\bar{V}\hat{\otimes} F^l$ for some $l\in\Zpluseq$, where $F^l\cong W(V_\half)$, that is the vertex subalgebra of $V$ generated by $V_\half$, and $\bar{V}:=\Com_V(W(V_\half))$. 
	Furthermore, $\bar{V}_\half=\{0\}$, both $\bar{V}$ and $F^l$ are nice and holomorphic.
\end{prop}

	$\bar{V}$ in Proposition \ref{prop:free-fermion-splitting} is called the \textit{stump} of $V$.
	The Free-Fermion Splitting allows us to solve the unitarity problem for nice holomorphic VOSAs with central charge $c=24$ only and then to obtain the full result as a corollary.

Now, recall the \textit{orbifold construction} for nice holomorphic VOAs with central charge $c\in 8\Zpluseq$ as given in \cite[Section 5]{EMS20}, see \cite[Section 3.3]{MS23} for a brief summary (to state the ``orbifold construction part'' of the following result, we will need the notion of VOA automorphisms \textit{of type $0$}, defined e.g.\ in \cite[p.\ 78]{EMS20}, but this is not important for our purposes).
Then the \textit{Neighbourhood Graph Method} says that: 

\begin{prop}[{\cite[Proposition 3.5 and Proposition 3.7]{HM}}]  \label{prop:neighbourhood_graph_method}
	Every nice holomorphic VOSA $V$ with central charge $c\in 8\Zpluseq$ and with $V_\parone\not=\{0\}$ can be realized by a couple $(U,g)$, where $U$ is a nice holomorphic VOA with central charge $c$ and $g$ is a VOA automorphism of $U$ of order $2$ and type $0$, as follows: the orbifold subalgebra $U^{\langle g\rangle}$ is the even part $V_\parzero$ of $V$; $V_\parzero$ is a nice VOA and $\Rep(V_\parzero)$ is a pointed even modular tensor category with fusion algebra $\C[\Z_2\times \Z_2]$; excluding $V_\parzero$ and $V_\parone$, the two remaining irreducible $V_\parzero$-modules have positive integer conformal weights and they give rise, as $\Z_2$-simple current extensions of $V_\parzero$, to $U$ and the orbifold VOA $U^{\mathrm{orb}(g)}$, which is still nice, holomorphic and with central charge $c$.
\end{prop}

\begin{rem}
	It may happen that $U$ and $U^{\mathrm{orb}(g)}$ in Proposition \ref{prop:neighbourhood_graph_method} are isomorphic. A sufficient, but not necessary, condition for that is that $V_\half\not=\{0\}$, see \cite[Proposition 3.6]{HM}.
\end{rem}

\begin{rem} \label{rem:holo_voas_cc_24}
	Recall that the nice holomorphic VOAs with central charge $c=24$, see e.g.\ \cite{MS23} and references therein, are proved to be unitary in \cite[Section 5]{CGGH23} and independently in \cite[Section 5]{Lam23}. 
	The \textit{moonshine VOA} $V^\natural$ and the \textit{Leech lattice VOA} $V_\Lambda$, see e.g.\ \cite{FLM88, Miy04}, are of this type.
	For smaller central charges, they are all realized as lattice VOAs \cite[Theorem 1 and Theorem 2]{DM04b}, which are known to be unitary, see Remark \ref{rem:lattice_vosas}.
\end{rem}

The last ingredient we need is a description of a nice holomorphic VOSA $V$ with central charge $c=24$ and with $V_\parone\not=\{0\}$ together with its even part $V_\parzero$ in terms of simple current extensions. To do that, we have to restrict ourselves to the case where the weight-$1$ subspace $V_1$ is non-trivial.  
It is known that $V_1$ is a reductive Lie algebra, see \cite[Theorem 1.1]{DM04}. Then let $H$ be the \textit{Heisenberg VOA} associated to a Cartan subalgebra of $V_1$, see e.g.\ \cite[Section 6.3]{LL04} and \cite[Section 4.3]{DL14} for its unitarity. 
Denote by $W$ the \textit{Heisenberg coset} $\Com_{V_\parzero}(H)$ which has conformal vector $\nu^W=\nu-\nu^H$, see Remark \ref{rem:coset_subalgebras}. 
Of course the central charge of $W$ is $24-\operatorname{rk}(V_1)$, where $\operatorname{rk}(V_1)$ is the dimension of the Cartan subalgebra of $V_1$.
By \cite[Theorem 1]{Mas14}, the double coset $\Com_{V_\parzero}(W)$ is a vertex subalgebra of $V_\parzero$ isomorphic to a lattice VOA $V_K$, where $K$ is a unique, up to isomorphism, positive-definite even lattice of rank $\operatorname{rk}(V_1)$. 
Accordingly, $K$ is called the \textit{associated lattice} of $V_\parzero$.
The couple $(W, V_K)$ is called a \textit{dual pair} (also Howe or commuting pair) in $V$ and their tensor product $W\otimes V_K$ is a full vertex subalgebra of $V_\parzero$.
By \cite[Section 4.3]{CKLR19}, $W$ is strongly rational and thus $\Rep(W)$ is a modular tensor category. Moreover, by \cite[Proposition 4.5]{HM}, $\Rep(W)$ is a pointed even modular tensor category equivalent to $\mathcal{C}(A_W)$ for some discriminant form $A_W$, see Remark \ref{rem:discriminant_form}. Consequently, if $A_K$ is the discriminant form associated to $V_K$, we have that $\Rep(W\otimes V_K)$ is equivalent to the Deligne product (see e.g.\ \cite[Section 1.11]{EGNO15}) $\mathcal{C}(A_W)\boxtimes \mathcal{C}(A_K)\cong \mathcal{C}(A_W\times A_K)$.
Starting from this fact, in \cite[Section 6]{HM}, the authors give an explicit description of every nice holomorphic VOSA with central charge $c=24$ and non-trivial weight-$1$ subspace together with its even part as a simple current extension of $W\otimes V_K$. 
In particular, every nice holomorphic VOSA with central charge $c=24$ and non-trivial weight-$1$ subspace falls into one of the three classes called \textit{of glueing type I}, \textit{II} (here there are two further subclasses called \textit{IIa} and \textit{IIb}) and \textit{III}, depending on how the simple currents in $\Rep(W\otimes V_K)$ glue together to give rise their VO(S)A extensions according to Proposition \ref{prop:neighbourhood_graph_method}.
In \cite[Section 7]{HM}, this description is combined with the orbifold construction given by Proposition \ref{prop:neighbourhood_graph_method} and it turns out that there are (at least, see the \textit{shorter moonshine uniqueness conjecture} in the Introduction for details) other two equivalence classes of nice holomorphic VOSAs with central charge $c=24$ and with $V_\parone\not=\{0\}$, but with trivial weight-$1$ subspace, which are ``non-typical'' VOSAs of glueing type III.
These are the graded tensor product $VB^\natural\hat{\otimes} F$ of the \textit{shorter moonshine VOSA} $VB^\natural$ \cite{Hoe95}, see also \cite[Section 4]{Yam05} and \cite[Section 1]{Hoe10}, with the real free fermion VOSA $F$ and the \textit{odd moonshine VOSA} $VO^\natural$ \cite{DGH88, Hua96}.
Their unitarity is treated separately in the following proposition.

\begin{prop} \label{prop:unitarity_moonshine_type_VOSAs}
	The tensor product $VB^\natural\hat{\otimes} F$ of the shorter moonshine VOSA and a real free fermion VOSA is unitary as well as the odd moonshine VOSA $VO^\natural$. Furthermore, their even parts are completely unitary VOAs. 
\end{prop}

\begin{proof}
	The unitarity of $VB^\natural\hat{\otimes} F$ follows from the unitarity of $VB^\natural$ proved in \cite[Theorem 7.4]{CGH} and Remark \ref{rem:graded_tensor_product}. To prove the complete unitarity of its even part, recall that the irreducible modules in its modular tensor category are all simple currents, see Proposition \ref{prop:neighbourhood_graph_method}. The possible simple current VOSA extensions are, up to isomorphism, either the moonshine VOA $V^\natural$ or $VB^\natural\hat{\otimes} F$, see \cite[Section 8.1]{HM}. As the moonshine VOA $V^\natural$ is unitary by \cite[Theorem 4.15]{DL14} and $VB^\natural\hat{\otimes} F$ is too as just proved, we have that those simple currents are all unitarizable, see Remark \ref{rem:unitarity_even_odd_parts} and references therein. It follows that $(VB^\natural\hat{\otimes} F)_\parzero$ is strongly unitary and thus it is completely unitary by Corollary \ref{cor:complete_unitarity_pointed_mtc}.
	
	Recall that the even part of the odd moonshine VOSA $VO^\natural$ is the fixed-point subalgebra $V_\Lambda^+$ of the Leech lattice VOA $V_\Lambda$ by its $(-1)$-involution $g$, see \cite{Hua96}. 
	Then the complete unitarity of $V_\Lambda^+$ follows from the unitarity of $V_\Lambda$ and of its irreducible $g$-twisted modules \cite[Theorem 4.12 and Theorem 4.14]{DL14} together with Theorem \ref{theo:complete_unitarity_holo_orbifold}. Then we can conclude that $VO^\natural$ is unitary, see Remark \ref{rem:unitarity_even_odd_parts}.
	An alternative proof can be given by recalling that $V_\Lambda^+$ is a framed VOA \cite[Eq.\ (6.7)]{DMZ94} and invoking \cite[Corollary 4.11 and Corollary 4.18]{CGGH23}.
\end{proof}

The following establishes the unitarity of every nice holomorphic VOSA $V$ with central charge $c=24$ and with $V_\parone\not=\{0\}\not= V_1$. Recall that the unitarity for the case with $V_\parone=\{0\}$ is already known, see Remark \ref{rem:holo_voas_cc_24}.

\begin{theo}  \label{theo:unitarity_holo_VOSAs_cc_24}
	Let $V$ be a nice holomorphic VOSA with central charge $c=24$ and with $V_\parone\not=\{0\}\not= V_1$. Then the tensor product $W\otimes V_K$ of the dual pair $(W,V_K)$ of $V_\parzero$ is a completely unitary VOA. In particular, $W$ and $V_\parzero$ are completely unitary VOAs and $V$ is a unitary VOSA.
\end{theo}

\begin{proof}
	
	By \cite[Theorem 8.2]{HM}, for every $V$, there exists an automorphism $\mu\in \mathrm{O}(\Lambda)\cong \mathrm{Co}_0$ (the largest of the Conway groups) of the Leech Lattice $\Lambda$ such that $W$ is isomorphic to the fixed-point subalgebra $V_{\Lambda_\mu}^{\langle\hat{\mu}\rangle}$, where $\Lambda_\mu$ is the coinvariant lattice of the fixed-point sublattice $\Lambda^\mu$, that is the orthogonal complement of $\Lambda^\mu$ in $\Lambda$, and $\hat{\mu}$ is a \textit{standard lift} of the isometry $\mu$ restricted to $\Lambda_\mu$ to $\Aut(V_{\Lambda_\mu})$. 
	We point out, see the beginning of \cite[Section 5.2]{Lam20}, that $V_{\Lambda_\mu}^{\langle\hat{\mu}\rangle}$ is a vertex subalgebra of the fixed-point subalgebra $V_\Lambda^{\langle\phi_\mu\rangle}$, where $\phi_\mu$ is a standard lift of the isometry $\mu$ of $\Lambda$ to $\Aut(V_\Lambda)$ and $\phi_\mu|_{\Lambda_\mu}=\hat{\mu}$. Moreover, $V_{\Lambda_\mu}^{\langle\hat{\mu}\rangle}$ and $V_{\Lambda^\mu}$ form a dual pair in $V_\Lambda^{\langle\phi_\mu\rangle}$. Recall that lattice VOAs are completely unitary, see Remark \ref{rem:lattice_vosas}.
	
	First, we prove that $V_\Lambda^{\langle\phi_\mu\rangle}$ is completely unitary and that $V_{\Lambda_\mu}^{\langle\hat{\mu}\rangle}$ is strongly unitary. 
	By \cite[Lemma 3.3]{Lam23}, $\phi_\mu$ is a unitary automorphism of $V_\Lambda$, cf.\ also \cite[Proposition A.5]{CGGH23} together with \cite{DN99}. 
	It follows that $V_\Lambda^{\langle\phi_\mu\rangle}$ is a unitary subalgebra of $V_\Lambda$, see Remark \ref{rem:fixed-point_subalgebras}, cf.\ \cite[Example 5.25]{CKLW18}. 
	In \cite[Section 3.2]{Lam23}, the author proves that the twisted modules of even lattice VOAs for any standard lift are unitary. Therefore, Theorem \ref{theo:complete_unitarity_holo_orbifold}, see also \cite[Remark 10.10]{Gui24}, says that $V_\Lambda^{\langle\phi_\mu\rangle}$ is completely unitary. 
	As $\Lambda^\mu$ is a sublattice of $\Lambda$, $V_{\Lambda^\mu}$ is a unitary subalgebra of $V_\Lambda$ and thus of $V_\Lambda^{\langle\phi_\mu\rangle}$ too, see Lemma \ref{lem:unitary_subalgebra_lattice_vosas}. 
	Then $V_{\Lambda_\mu}^{\langle\hat{\mu}\rangle}$ is also a unitary subalgebra of $V_\Lambda^{\langle\phi_\mu\rangle}$, as it is the coset subalgebra of $V_{\Lambda^\mu}$ in $V_\Lambda^{\langle\phi_\mu\rangle}$, see Remark \ref{rem:coset_subalgebras}, cf.\ \cite[Example 5.27]{CKLW18}.
	Recall that $W\cong V_{\Lambda_\mu}^{\langle\hat{\mu}\rangle}$ is strongly rational by \cite[Section 4.3]{CKLR19} and thus we are in charge to apply \cite[Theorem 2]{KM15} which says that all irreducible $V_{\Lambda_\mu}^{\langle\hat{\mu}\rangle}$-modules appear in some irreducible $V_\Lambda^{\langle\phi_\mu\rangle}$-module. It follows from the complete unitarity of the latter that $V_{\Lambda_\mu}^{\langle\hat{\mu}\rangle}$ is strongly unitary, that is every $V_{\Lambda_\mu}^{\langle\hat{\mu}\rangle}$-module is unitarizable, cf.\ the beginning of the proof of \cite[Theorem 2.6.5]{Gui}. This implies that $W$ is strongly unitary too.
	
	Now, recall that $W$ is strongly rational and that the irreducible modules in the modular tensor category $\Rep(W)$ are all simple currents, that is $\Rep(W)$ is pointed, see also the beginning of \cite[Section 6.1]{HM}. By Corollary \ref{cor:complete_unitarity_pointed_mtc}, it follows that $W$ is completely unitary, that is $\Repu(W)$ is a unitary modular tensor category. $V_K$ is also completely unitary, so that $W\otimes V_K$ is completely unitary too, see Remark \ref{rem:graded_tensor_product}. Therefore $V_\parzero$ is a completely unitary VOA and $V$ is a unitary VOSA by Theorem \ref{theo:unitarity_vosa_extensions}.
\end{proof}

\begin{cor}  \label{cor:unitarity_holo_c<24}
	Let $V$ be a nice holomorphic VOSA with central charge $c<24$. Suppose that $V$ is not a VOSA with central charge $c=\frac{47}{2}$ and with $V_\half=\{0\}=V_1$, unless it is isomorphic to $VB^\natural$. Then $V$ is unitary.
\end{cor}

\begin{proof}
	By Proposition \ref{prop:free-fermion-splitting}, there exists $l\in \Zplus$ such that $V\hat{\otimes} F^l$ is a nice holomorphic VOSA with central charge $c+\frac{l}{2}=24$, say $U$, such that $U_\half\not=\{0\}\not=U_\parone$. 
	If $U_1=\{0\}$, then $l=1$ and $V_\half=\{0\}=V_1$, so that $V\cong VB^\natural$ and $U\cong VB^\natural\hat{\otimes}F$. $VB^\natural$ is already known to be unitary thanks to \cite[Theorem 7.24]{CGH}. Then we can suppose that $U_1\not=\{0\}$, so that $U$ is unitary by Theorem \ref{theo:unitarity_holo_VOSAs_cc_24}.
	Trivially $U_\half$ is PCT-invariant, so that $F^l$ is PCT-invariant too. Therefore $F^l$ is a unitary subalgebra of $U$. The coset subalgebra of $F^l$ in $U$ is $V$ itself. Then $V$ is a unitary subalgebra of $U$ and thus it has a structure of unitary VOSA, see Remark \ref{rem:coset_subalgebras}. 
\end{proof}

\subsection{Strong graded locality}
\label{subsec:gl_conformal_nets}

We use the notation introduced in Section \ref{subsec:unitarity} and we refer to Section \ref{sec:preliminaries} for the necessary preliminaries. 
Recall that the free fermion VOSAs $F^l$ for all $l\in\Zplus$ are strongly graded-local by \cite[Example 7.2]{CGH}. The corresponding \textit{free fermions nets} are denoted by $\F^l$ respectively. Of course, the \textit{real free fermion net} $\F^1$ is simply denoted by $\F$. Moreover, it is well-known that $F_\parzero$ is isomorphic to the Virasoro VOA $L(\half,0)$ and $F_\parone$ to the $L(\half,0)$-module $L(\half,\half)$. Accordingly, the Bose part $\F_\parzero$ of $\F$ is the \textit{Virasoro net} $\Vir_\half:=\A_{L(\half,0)}$, see Remark \ref{rem:examples_strong_braiding_complete_energy_boundedness}.

\begin{prop} \label{prop:holo_vosas_strong_braiding}
Let $V$ be a nice holomorphic VOSA with central charge $c=24$ and with $V_\parone\not=\{0\}$.
Assume one of the following further hypotheses.
\begin{itemize}
	\item[(i)] Suppose that $V$ is of glueing type I or II. Then consider the dual pair $(W,V_K)$ of $V_\parzero$ and set $E:=W$ and $X:=V_K$.
	
	\item[(ii)] Suppose that $V$ has $V_\half\not=\{0\}\not=V_1$ or that $V$ is isomorphic to $VB^\natural\hat{\otimes}F$.
	Then consider the Free-Fermion Splitting $\bar{V}\hat{\otimes} F^l$ of $V$ for some $l\in\Zplus$ and set $E:=(\bar{V}\hat{\otimes}F^{l-1})_\parzero$ and $X:=F_\parzero$, identified with the Virasoro VOA $L(\half,0)$. 
\end{itemize}
Then $E$ and the tensor product $E\otimes X$ are completely unitary, strongly local, completely energy-bounded and every unitary intertwining operator satisfies the strong intertwining property and the strong braiding. In particular, the strong integrability $^*$-functor $\mathfrak{F}:\Repu(E\otimes X)\to \Repf(\A_E\otimes \A_X)$ realizes an equivalence of unitary modular tensor categories, so that $\A_E\otimes \A_X$ is completely rational. 
Furthermore, $V$ is energy-bounded and there exists a holomorphic graded-local conformal net extension $\A_{E\otimes X}^V$ of $\A_E\otimes \A_X$ with non-trivial Fermi part and whose Bose subnet is the irreducible conformal net $\A_{V_\parzero}$.
If $\B$ is any irreducible graded-local conformal net with non-trivial Fermi part and whose Bose subnet is isomorphic to $\A_{V_\parzero}$, then $\B$ is isomorphic to $\A_{E\otimes X}^V$.
\end{prop}

\begin{proof}
\textit{The case (i)}.
As showed in \cite[Section 6.1]{HM} and recalled in Section \ref{subsec:unitarity}, if $V$ is of glueing type I or II, then $W$ forms a dual pair with an even lattice VOA $V_L$, for a given lattice $L$, in a strongly rational holomorphic VOA $U$ with central charge $c=24$, whose strong locality is proved in \cite[Theorem 5.5]{CGGH23}.
Trivially, $U$ satisfies the hypotheses of Theorem \ref{theo:strong_braiding_complete_energy_boundedness} and it is also completely energy-bounded. The same holds for lattice VOAs, see Remark \ref{rem:examples_strong_braiding_complete_energy_boundedness}. Recall that $W$ is strongly rational by \cite[Section 4.3]{CKLR19} and it is also completely unitary by Theorem \ref{theo:unitarity_holo_VOSAs_cc_24}. 
It follows from Remark \ref{rem:graded_tensor_product} that $W\otimes V_L$ is completely unitary too and thus it can be made a unitary subalgebra of $U$ by Theorem \ref{theo:unitarity_vosa_extensions}.
Then also $W$ and $V_L$ are unitary subalgebras of $U$. 
By (iv) of Theorem \ref{theo:strong_braiding_complete_energy_boundedness}, $W=\Com_U(V_L)$ is completely energy-bounded and its unitary intertwining operators satisfy the strong intertwining property and the strong braiding. By (v) of Theorem \ref{theo:strong_braiding_complete_energy_boundedness}, the same holds for $W\otimes V_K$ (of course, also $W\otimes V_L$ does so). Note that $W$ and $W\otimes V_K$ are also strongly local and that $\A_W\otimes \A_{V_K}=\A_{W\otimes V_K}$, see Remark \ref{rem:properties_strongly_graded-local_VOSAs}.

By (ii) of Theorem \ref{theo:strong_braiding_complete_energy_boundedness}, the strong integrability $^*$-functor $\mathfrak{F}:\Repu(W\otimes V_K)\to \Rep(\A_W\otimes \A_{V_K})$ is fully faithful and such that $\Repu(W\otimes V_K)$ is equivalent to a braided $C^*$-tensor subcategory of $\Rep(\A_W\otimes \A_{V_K})$.
With the help of \cite[Corollary 4.8]{CGGH23}, we choose the unitary structures in such a way that $W\otimes V_K\subset V_\parzero\subset U$ are inclusions of unitary subalgebras.
Therefore, we can apply \cite[Theorem 2.7.2]{Gui} to the inclusion $W\otimes V_K\subset U$, to conclude that $\Repu(W\otimes V_K)$ and $\Repf(\A_W\otimes \A_{V_K})$ have the same number of irreducibles, as $U$ is holomorphic and $\A_U$ is too by \cite[Theorem 5.5]{CGGH23}. Then $\A_W\otimes \A_{V_K}$ is completely rational and $\mathfrak{F}$ realizes an equivalence of unitary modular tensor categories between $\Repu(W\otimes V_K)$ and $\Repf(\A_W\otimes \A_{V_K})$. 

Now, $\A_{W\otimes V_K}^V$ is realized by two steps.
First, we apply \cite[Theorem 4.10]{CGGH23}, which works as follows.
Consider the simple CFT type VOA extension $W\otimes V_K\subset V_\parzero$ in $\Repu(W\otimes V_K)$, which is equivalent to $\Repf(\A_W\otimes \A_{V_K})$. Then this equivalence induces an irreducible conformal net extension $\A_{W\otimes V_K}^{V_\parzero}$ of $\A_{W\otimes V_K}$.
On the other hand, the irreducible conformal net extension $\A_{W\otimes V_K}\subset \A_{V_\parzero}$ induces a simple CFT type VOA extension $W\otimes V_K\subset \widetilde{V}$ via the strong integrability $^*$-functor $\mathfrak{F}$, which implies that $\widetilde{V}$ must be equivalent to $V_\parzero$ as $W\otimes V_K$-module. In particular, $\widetilde{V}$ is a simple current extension of $W\otimes V_K$ and thus it must be actually isomorphic to $V_\parzero$ as VOAs by the uniqueness of the simple current extension structure, see \cite[Proposition 5.3]{DM04}.
This implies that $\A_{W\otimes V_K}^{V_\parzero}$ is isomorphic to $\A_{V_\parzero}$ and that $\Repu(V_\parzero)$ and $\Repf(\A_{V_\parzero})$ are equivalent as unitary modular tensor category. By \cite[Theorem 24]{Lon03}, $\A_{V_\parzero}$ is also completely rational.
Second, we apply \cite[Theorem 4.21]{CGGH23}, so that $\A_{W\otimes V_K}^V$ is obtained as irreducible graded-local conformal net extension of $\A_{V_\parzero}$, induced by the simple CFT type VOSA extension $V$ of $V_\parzero$.

To prove that $\A_{W\otimes V_K}^V$ is holomorphic, note that it is completely rational as its Bose part $\A_{V_\parzero}$ is. By Proposition \ref{prop:neighbourhood_graph_method} and the equivalence of $\Repu(V_\parzero)$ and $\Repf(\A_{V_\parzero})$, we have that $\A_{V_\parzero}$ counts four irreducible representations. Of course, one of it is the vacuum representation, whereas the other three representations are $\Z_2$-simple currents with trivial twist, apart from the one giving rise to the extension $\A_{W\otimes V_K}^V$, say $\sigma$, whose twist is $-\one_\sigma$. 
By \cite[Proposition 22]{CKL08}, if $\pi$ is any irreducible (Neveu-Schawrz) representation of $\A_{W\otimes V_K}^V$, then it splits into two inequivalent irreducible representations of $\A_{V_\parzero}$, say $\pi_+$ and $\pi_-:=\pi_+\sigma$.
Moreover, $\pi_+$ and $\pi_-$ are both \textit{$\sigma$-Bose} in the meaning of \cite[Section 2.5]{CKL08} thanks to \cite[Corollary 25]{CKL08}.
If the \textit{statistics phase} $\omega_{\pi_+}$, which is the complex number obtained from the twist of $\pi_+$ in $\Repf(\A_{V_\parzero})$, is equal to $-1$, then $\pi_+=\sigma$ and $\pi_-$ must be the vacuum representation of $\A_{V_\parzero}$.
If $\omega_{\pi_+}=1$, then $\omega_{\pi_-}=-\omega_{\pi_+}=-1$ according to \cite[Eq.\ (6)]{CKL08}, so that $\pi_-=\sigma$ and $\pi_+$ is the vacuum representation of $\A_{V_\parzero}$. By \cite[Proposition 2.5]{CHKLX15}, $\pi$ is the vacuum representation of $\A_{W\otimes V_K}^V$, that is $\A_{W\otimes V_K}^V$ is holomorphic. 

For the uniqueness statement, let $\B$ be any irreducible graded-local conformal net with non-trivial Fermi part and whose Bose subnet is isomorphic to $\A_{V_\parzero}$.
Again by \cite[Theorem 4.21]{CGGH23} and its proof, this corresponds to a $\Z_2$-simple current extension $V_\parzero\subset \widehat{V}$, where the odd part $\widehat{V}_\parone$ is the $\Z_2$-simple current in the representation category of $V_\parzero$. By Proposition \ref{prop:neighbourhood_graph_method}, $V_\parzero$ has a unique $\Z_2$-simple current with semi-integer conformal weight and thus $\widehat{V}_\parone$ must be isomorphic to $V_\parone$ as $\Z_2$-simple current of $V_\parzero$.
Moreover, $\widehat{V}\cong V$ as the VOSA structure of a simple current extension is unique up to isomorphism, see \cite[Theorem 3.9]{CKL20} and \cite[Proposition 5.3]{DM04}. Then the uniqueness statement follows from the construction of $\A_{W\otimes V_K}^V$.

Finally, the energy boundedness of $V$ is a consequence of the complete energy boundedness of $W\otimes V_K$. Indeed, as $V$ is a $W\otimes V_K$-module, every vertex operator $Y(a,z)$ with $a\in V$ can be considered as a finite sum of unitary intertwining operators of $W\otimes V_K$, so that the energy bounds for $Y(a,z)$ follow, cf.\ \cite[Theorem 4.6]{CT23}.

\textit{The case (ii)}.
Suppose that $V\not\cong VB^\natural\hat{\otimes}F$, so that $\bar{V}\not\cong VB^\natural$. 
We safely identify $V$ with its Free-Fermion Splitting $\bar{V}\hat{\otimes}F^l$ for some $l\in\Zplus$, see Proposition \ref{prop:free-fermion-splitting}.  
Set $T:=\bar{V}\hat{\otimes}F^{l-1}$, which is nice, holomorphic and with central charge equal to $\frac{47}{2}$.
By Remark \ref{rem:niceness_even_odd_parts}, $T_\parzero$ is nice too.
Moreover, by \cite[Proposition 3.2 and Table 1]{HM}, the fusion algebra of $\Rep(T_\parzero)$ is of Ising type with only three elements $\id$, $\sigma$ and $\tau$, whose corresponding conformal weights modulo $\Z$ are $0$, $\half$ and $\frac{15}{16}$. Note also that the simple CFT type VOSA extensions $T_\parzero\otimes L(\half,0)\subset (T\hat{\otimes} F)_\parzero$ and $T_\parzero\otimes L(\half,0)\subset V$ are made by simple currents only. Indeed, the decomposition of $V$ into irreducible $(T_\parzero\otimes L(\half,0))$-modules is
\begin{equation}  \label{eq:simple_currents_dual_pair_type_III_fermion}
	V= \left(T_\parzero\otimes L(\half,0)\right) \oplus
	   \left(T_\parzero\otimes L(\half,\half) \right) \oplus
	   \left(T_\parone\otimes L(\half,0)\right) \oplus
	   \left(T_\parone\otimes L(\half,\half)\right) \,.
\end{equation}

By Proposition \ref{prop:neighbourhood_graph_method}, we have inclusions of vertex subalgebras $T_\parzero\otimes L(\half,0)\subset (T\hat{\otimes} F)_\parzero\subset U$ for some strongly rational holomorphic VOA $U$ with central charge $c=24$ and $U_1\not=\{0\}$. Note that $\Com_U(L(\half,0))$ must be a simple CFT type VOA extension of $T_\parzero$ and thus we can consider its decomposition into irreducible $T_\parzero$-modules. If we look at the possible conformal weights of these modules, the only possibility is that $\Com_U(L(\half,0))=T_\parzero$ (actually, $(T_\parzero, L(\half,0))$ turns out to be a dual pair for both $U$ and $V_\parzero$).

By Corollary \ref{cor:unitarity_holo_c<24}, $T$ is unitary. By Remark \ref{rem:graded_tensor_product}, if we equip $V$ with the unitary structure induced by the graded tensor product $T\hat{\otimes} F$, then $T_\parzero\otimes L(\half,0)$ turns out to be a unitary subalgebra of $(T\hat{\otimes} F)_\parzero= V_\parzero$ and of $V$ too. Recall that $V_\parzero$ is the fixed-point subalgebra of $U$ by some automorphism of order $2$. Recall also that $U$ is unitary, see Remark \ref{rem:holo_voas_cc_24}. By \cite{GN03, DZ08}, $U$ is finitely generated as it is strongly rational. By \cite[Proposition A.5]{CGGH23}, we can choose the unitary structure on $U$ in such a way that $V_\parzero$ is a unitary subalgebra of $U$ and thus also $T_\parzero$, $L(\half,0)$ and their tensor product are.

Recall that $L(\half, 0)$ satisfies the hypotheses of Theorem \ref{theo:strong_braiding_complete_energy_boundedness}, see Remark \ref{rem:examples_strong_braiding_complete_energy_boundedness}.
Then the conclusion follows noting that now we have the same conditions as in the case (i) for $T_\parzero$ and $L(\half,0)$ here in place of $W$ and $V_K$ there. 
Finally, if $V\cong VB^\natural\hat{\otimes}F$, then most of the desired statements have been already given in the proof of \cite[Theorem 7.24]{CGH} in a very similar way, but using as $U$ here the moonshine VOA $V^\natural$ there. Then the missing parts follow from the argument above.
\end{proof}

\begin{rem}  \label{rem:uniqueness_net_extensions_double_pairs}
	Note that if $V$ satisfies both (i) and (ii) of Proposition 3.8, then we have two different ways to produce a holomorphic graded-local conformal net with non-trivial Fermi part and whose Bose subnet is $\A_{V_\parzero}$. Nevertheless, by the uniqueness statement there, these two procedures bring to the same holomorphic graded-local conformal net.
\end{rem}

\begin{theo}  \label{theo:strong_graded_locality_holo}
	Let $V$ be a nice holomorphic VOSA with central charge $c=24$ and with $V_\parone\not=\{0\}\not=V_1$ or be isomorphic to $VB^\natural\hat{\otimes}F$. If $V$ is not a VOSA of glueing type III and with $V_\half=\{0\}$, then $V$ is strongly graded-local and $\A_V$ is holomorphic.
\end{theo}

\begin{proof}
	Note that $V$ must fall in at least one of the cases (i) or (ii) of Proposition \ref{prop:holo_vosas_strong_braiding}.
	Therefore $V$ is energy-bounded. 
	Moreover, the tensor product $E\otimes X$ in Proposition \ref{prop:holo_vosas_strong_braiding} is a strongly local completely unitary VOA and its unitary intertwining operators satisfy the strong intertwining property and the strong braiding.
	By \cite[Section 6.1]{HM}, cf.\ also Lemma \ref{lem:simple_current_extensions}, we can see that if $V$ is of glueing type I or II, then it is a simple current extension of the VOA $W\otimes V_K$. If $V$ is in the case (ii) of Proposition \ref{prop:holo_vosas_strong_braiding}, then it is a simple current extension of the VOA $T_\parzero\otimes L(\half,0)$, see \eqref{eq:simple_currents_dual_pair_type_III_fermion}. 
	Therefore, $E\otimes X$ satisfies the hypotheses of Lemma \ref{lem:strong_graded_locality_simple_current_ext}, so that $V$ is strongly graded-local.
	Finally, $(\A_V)_\parzero=\A_{V_\parzero}$ by Remark \ref{rem:properties_strongly_graded-local_VOSAs} and thus $\A_V$ is isomorphic to $\A_{E\otimes X}^V$ by the uniqueness statement of Proposition \ref{prop:holo_vosas_strong_braiding}, see also Remark \ref{rem:uniqueness_net_extensions_double_pairs}, implying that $\A_V$ is holomorphic. Note that the strong graded locality of $VB^\natural\hat{\otimes}F$ was already proved by \cite[Theorem 7.24 and Corollary 6.6]{CGH} with similar technique.
\end{proof}

Strongly rational holomorphic VOSAs with central charge $c\leq 12$ were explicitly classified in \cite[Theorem 3.1]{CDR18} and proved to be strongly graded-local in \cite[Theorem 7.22]{CGH}. We improve the latter result with the following:

\begin{cor}  \label{cor:strong_graded_locality_c_less_24}
	Let $V$ be a nice holomorphic VOSA with central charge $c<24$. Suppose that $V$ is not a VOSA with central charge $c=\frac{47}{2}$ and with $V_\half=\{0\}=V_1$, unless it is isomorphic to $VB^\natural$. Then $V$ is strongly graded-local. In particular, it gives rise to a holomorphic graded-local conformal net $\A_V$.
\end{cor}

\begin{proof}
	By Proposition \ref{prop:free-fermion-splitting}, there exists $l\in \Zplus$ such that $V\hat{\otimes} F^l$ is a nice holomorphic VOSA with central charge $c+\frac{l}{2}=24$, say $U$, such that $U_\half\not=\{0\}\not=U_\parone$. 
	If $U_1=\{0\}$, then $l=1$ and $V_\half=\{0\}=V_1$, so that $V\cong VB^\natural$ and $U\cong VB^\natural\hat{\otimes}F$.
	Therefore $U$ is strongly graded-local by Theorem \ref{theo:strong_graded_locality_holo}. By Corollary \ref{cor:unitarity_holo_c<24}, Remark \ref{rem:graded_tensor_product} and the unitarity of $F$, it follows that $V$ is a unitary subalgebra of $U$ and thus it is strongly graded-local thanks to Remark \ref{rem:properties_strongly_graded-local_VOSAs}.  
	If $\A_V$ is not holomorphic, then we can construct, see \cite[Section 2.6]{CKL08}, at least one non-trivial irreducible representation on the irreducible graded-local conformal net $\A_V\hat{\otimes} \F^l$. This leads to a contradiction as $\A_V\hat{\otimes} \F^l$ is isomorphic to $\A_U$, which is instead holomorphic. 
\end{proof}

\begin{rem}
	Note that Corollary \ref{cor:strong_graded_locality_c_less_24} proves that the \textit{super-moonshine net} $\A^{f\natural}:=\A_{V^{f\natural}}$, already defined in \cite[Theorem 7.21]{CGH} from the \textit{super-moonshine VOSA} $V^{f\natural}$, see \cite{Dun07} and also \cite{DM15}, is indeed holomorphic as the latter has central charge $c=12$.
	Moreover, it also proves the holomorphicity of the \textit{shorter moonshine net} $\A_{VB^\natural}$, already defined in \cite[Theorem 7.24]{CGH}.
\end{rem}

\begin{rem}
	One could wonder why the same technique used to prove the strong graded-locality for nice holomorphic VOSAs with central charge $c=24$ and of glueing type I and II does not work for the ones of glueing type III. The reason is that in the latter case, the Heisenberg coset $W$ does not form a dual pair with a lattice VOA in one of the holomorphic VOAs obtained from the Neighbourhood Graph Method, as we can see from \cite[Section 6.1]{HM}. In turn, this prevents us to apply (iv) of Theorem \ref{theo:strong_braiding_complete_energy_boundedness} in the proof of the case (i) of Proposition \ref{prop:holo_vosas_strong_braiding}, which gives the complete energy boundedness of $W$, the strong intertwining property and the strong braiding for its unitary intertwining operators.
\end{rem}

\begin{con}   \label{con:strong_graded_locality_holo_VOSAs}
Every nice holomorphic VOSA $V$ with central charge $c\leq 24$ is strongly graded-local. In particular, it gives rise to a holomorphic graded-local conformal net $\A_V$.
\end{con}

\bigskip
\noindent
{\small
	{\bf Acknowledgements.}
	The author would like to sincerely thank S.\ Carpi and R.\ Hillier for the constant support, for reading and giving useful comments on a preliminary version of this work. 
	He is extremely grateful to S.\ Carpi for the very helpful and inspiring discussions. 
	He would also like to thank S.\ Carpi and L.\ Giorgetti for their hospitality at the Department of Mathematics of the University of Rome ``Tor Vergata''.
	The author is supported by the Leverhulme Trust Research Project Grant RPG-2021-129.


\begin{thebibliography}{00}
	
	\bibitem[ABD04]{ABD04}
	T. Abe, G. Buhl and C. Dong.
	Rationality, regularity, and $C_2$-cofiniteness.
	\emph{Trans. Amer. Math. Soc.} 
	\textbf{356}(8), 3391--3402 (2004).
	
	\bibitem[AL17]{AL17}
	C. Ai and X. Lin. 
	On the unitary structures of vertex operator superalgebras.
	\emph{J. Alg.}
	\textbf{487}, 217--243 (2017).
	
	\bibitem[Ara10]{Ara10}
	H. Araki.
	\emph{Mathematical Theory of Quantum Fields}.
	(\emph{Internat. Ser. Monogr. Phys.} \textbf{101}, Oxford University Press Inc., New York, 2010).
	
	\bibitem[BK04]{BK04}
	B. Bakalov and V. G. Kac.
	Twisted modules over lattice vertex algebras.
	\emph{In:} Lie Theory and Its Applications in Physics V.
	(World Scientific Publishing Co., Inc., River Edge, NJ, 2004), 3--26.
	
	\bibitem[BK01]{BK01}
	B. Bakalov and A. Jr. Kirillov.
	\emph{Lectures on Tensor Categories and Modular Functors}.
	(\emph{Univ. Lecture Ser.} \textbf{21},
	American Mathematical Society, Providence, RI, 2001).
	
	\bibitem[BLS23]{BLS23}
	K. Betsumiya, C. H. Lam and H. Shimakura.
	Automorphism groups and uniqueness of holomorphic vertex operator algebras of central charge 24.
	\emph{Comm. Math. Phys.}
	\textbf{399}(3), 1773--1810 (2023).
	
	\bibitem[BLS24]{BLS24}
	K. Betsumiya, C. H. Lam and H. Shimakura.
	Automorphism groups of cyclic orbifold vertex operator algebras associated with the Leech lattice and some non-prime isometries.
	\emph{Israel J. Math.}
	\textbf{259}(2), 621--650 (2024).
	
	\bibitem[BLTZ24]{BLTZ24}
	P. Boyle Smith, Y.-H. Lin, Y. Tachikawa and Y. Zheng.
	Classification of chiral fermionic CFTs of central charge $\leq 16$.
	\emph{SciPost Phys.}
	\textbf{16}(2), Paper No.\ 058, pp.\ 36 (2024).
	
	\bibitem[CM]{CM}
	S. Carnahan and M. Miyamoto.
	Regularity of fixed-point vertex operator subalgebras. \href{https://doi.org/10.48550/arXiv.1603.05645}{\emph{arXiv:1603.05645v4} \textbf{[math.RT]}}. (2016).
	
	\bibitem[Car04]{Car04}
	S. Carpi.
	On the representation theory of Virasoro nets.
	\emph{Commun. Math. Phys.}
	\textbf{244}(2), 261--284 (2004).
		
	\bibitem[CGGH23]{CGGH23}
	S. Carpi, T. Gaudio, L. Giorgetti and R. Hillier.
	Haploid algebras in $C^*$-tensor categories and the Schellekens list.
	\emph{Comm. Math. Phys.}
	\textbf{402}(1), 169--212 (2023).
	
	\bibitem[CGH19]{CGH19}
	S. Carpi, T. Gaudio and R. Hillier.
	Classification of unitary vertex subalgebras and conformal subnets for rank-one lattice chiral CFT models.
	\emph{J. Math. Phys.}
	\textbf{60}(9), 093505, pp.\  20 (2019).
	
	\bibitem[CGH]{CGH}
	S. Carpi, T. Gaudio and R. Hillier.
	From vertex operator superalgebras to graded-local conformal nets and back.
	\href{https://doi.org/10.48550/arXiv.2304.14263}{\textit{arXiv:2304.14263v1} \textbf{[math.OA]}}. (2023).
	
	\bibitem[CHKLX15]{CHKLX15}
	S. Carpi, R. Hillier, Y. Kawahigashi, R. Longo and F. Xu.
	$N=2$ superconformal nets.
	\emph{Commun. Math Phys.}
	\textbf{336}(3), 1285--1328 (2015).
	
	\bibitem[CHL15]{CHL15}
	S. Carpi, R. Hillier and R. Longo.
	Superconformal nets and noncommutative geometry.
	\emph{J. Noncommut. Geom.}
	\textbf{9}(2), 391--445 (2015).
	
	\bibitem[CKL08]{CKL08}
	S. Carpi, Y. Kawahigashi and R. Longo.
	Structure and classification of superconformal nets.
	\emph{Ann. Henri Poincaré} 
	\textbf{9}(6), 1069–1121 (2008).
	
	\bibitem[CKLW18]{CKLW18}
	S. Carpi, Y. Kawahigashi, R. Longo and M. Weiner.
	From vertex operator algebras to conformal nets and back.
	\emph{Mem. Amer. Math. Soc.} 
	\textbf{254}(1213), pp.\  vi+85 (2018).
	
	\bibitem[CRTT]{CRTT}
	S. Carpi, C. Raymond, Y. Tanimoto and J. E. Tener.
	Non-unitary Wightman CFTs and non-unitary vertex algebras.
	\href{https://doi.org/10.48550/arXiv.2409.08454}{\textit{arXiv:2409.08454v1} \textbf{[math-ph]}}. (2024).
	
	\bibitem[CTW22]{CTW22}
	S. Carpi, Y. Tanimoto and M. Weiner.
	Local energy bounds and strong locality in chiral CFT.
	\emph{Comm. Math. Phys.}
	\textbf{390}(1), 169--192 (2022).
	
	\bibitem[CTW23]{CTW23}
	S. Carpi, Y. Tanimoto and M. Weiner.
	Unitary representations of the $\mathcal{W}_3$-algebra with $c\geq2$.
	\emph{Transform. Groups}
	\textbf{28}(2), 561--590 (2023).
	
	\bibitem[CT23]{CT23}
	S. Carpi and L. Tomassini.
	Energy bounds for vertex operator algebra extensions. 
	\emph{Lett. Math. Phys.}
	\textbf{113}(3), Paper No.\ 59, pp.\ 24 (2023).
	
	\bibitem[CWX]{CWX}
	S. Carpi, M. Weiner and F. Xu. 
	From vertex operator algebra modules to representations of conformal nets. 
	\emph{In preparation}.
	
	\bibitem[CLM22]{CLM22}
	N. Chigira, C. H. Lam and M. Miyamoto.
	Orbifold construction and Lorentzian construction of Leech lattice vertex operator algebra
	\emph{J. Algebra}
	\textbf{593}
	(2022), 26--71.
	
	\bibitem[CDR18]{CDR18}
	T. Creutzig, J. F. R. Duncan and W. Riedler.
	Self-dual vertex operator superalgebras and superconformal field theory. 
	\emph{J. Phys. A} 
	\textbf{51}(3), 034001, pp.\ 29 (2018).
	
	\bibitem[CKL20]{CKL20}
	T. Creutzig, S. Kanade and A. R. Linshaw.
	Simple current extensions beyond semi-simplicity.
	\emph{Commun. Contemp. Math.} 
	\textbf{22}(1), 1950001, pp.\ 49 (2020).
	
	\bibitem[CKLR19]{CKLR19}
	T. Creutzig, S. Kanade, A. R. Linshaw and D. Ridout.
	Schur-Weyl duality for Heisenberg cosets.
	\emph{Transform. Groups}
	\textbf{24}(2), 301--354 (2019).
	
	\bibitem[CKM24]{CKM24}
	T. Creutzig, S. Kanade and R. McRae.
	Tensor categories for vertex operator superalgebra extensions.
	\emph{Mem. Amer. Math. Soc.}
	\textbf{295}(1472), pp.\ vi+181  (2024).
	
	\bibitem[DGH88]{DGH88}
	L. Dixon, P. Ginsparg and J. Harvey.
	Beauty and the beast: superconformal symmetry in a Monster module.
	\emph{Comm. Math. Phys.}
	\textbf{119}(2), 221--241 (1988).
	
	\bibitem[DJX13]{DJX13}
	C. Dong, X. Jiao and F. Xu.
	Quantum dimensions and quantum Galois theory.
	\emph{Trans. Amer. Math. Soc.}
	\textbf{365}(12), 6441--6469 (2013).
	
	\bibitem[DL93]{DL93}
	C. Dong and J. Lepowsky.
	\emph{Generalized Vertex Algebras and Relative Vertex Operators}.
	(\emph{Progr. Math.} \textbf{112}, Birkh{\"a}user Boston, Inc., Boston, MA, 1993).
	
	\bibitem[DLM97]{DLM97}
	C. Dong, H. Li and G. Mason.
	Regularity of rational vertex operator algebras.
	\emph{Adv. Math.}
	\textbf{132}(1), 148--166 (1997).
	
	\bibitem[DLM00]{DLM00}
	C. Dong, H. Li and G. Mason.
	Modular-invariance of trace functions in orbifold theory and generalized Moonshine.
	\emph{Comm. Math. Phys.}
	\textbf{214}(1), 1--56 (2000).
	
	\bibitem[DL14]{DL14}
	C. Dong and X. Lin.
	Unitary vertex operator algebras.
	\emph{J. Algebra} 
	\textbf{397}, 252--277 (2014).
	
	\bibitem[DLN15]{DLN15}
	C. Dong, X. Lin and S.-H. Ng.
	Congruence property in conformal field theory.
	\emph{Algebra Number Theory}
	\textbf{9}(9), 2121--2166 (2015).
	
	\bibitem[DM04]{DM04}
	C. Dong and G. Mason.
	Rational vertex operator algebras and the effective central charge.
	\emph{Int. Math. Res. Not.}
	\textbf{2004}(56), 2989--3008 (2004).
	
	\bibitem[DM04b]{DM04b}
	C. Dong and G. Mason.
	Holomorphic vertex operator algebras of small central charge.
	\emph{Pacific J. Math.}
	\textbf{213}(2), 253--266 (2004).
	
	\bibitem[DMZ94]{DMZ94}
	C. Dong, G. Mason and Y. Zhu.
	Discrete series of the Virasoro algebra and the moonshine module.
	\emph{In:}
	Algebraic Groups and Their Generalizations: Quantum and Infinite-Dimensional Methods. 
	(\emph{Proc. Sympos. Pure Math.} \textbf{56}(Part 2),
	American Mathematical Society, Providence, RI, 1994), 295--316. 
	
	\bibitem[DN99]{DN99}
	C. Dong and K. Nagatomo.
	Automorphism groups and twisted modules for lattice vertex operator algebras.
	\emph{In:} 
	Recent Developments in Quantum Affine Algebras and Related Topics.
	(\emph{Contemp. Math.}	\textbf{248},  American Mathematical Society, Providence, 1999), 117--133.
	
	\bibitem[DNR21]{DNR21}
	C. Dong, S.-H. Ng and L. Ren.
	Vertex operator superalgebras and the 16-fold way.
	\emph{Trans. Amer. Math. Soc.}
	\textbf{374}(11), 7779--7810 (2021).
	
	\bibitem[DRY22]{DRY22}
	C. Dong, L. Ren and M. Yang.
	Super orbifold theory.
	\emph{Adv. Math.}
	\textbf{405}, Paper No.\ 108481, pp.\ 34  (2022).
	
	\bibitem[DZ08]{DZ08}
	C. Dong and W. Zhang.
	Rational vertex operator algebras are finitely generated.
	\emph{J. Algebra}
	\textbf{320}(6), 2610--2614 (2008).
	
	\bibitem[DZ05]{DZ05}
	C. Dong and Z. Zhao.
	Modularity in orbifold theory for vertex operator superalgebras.
	\emph{Comm. Math. Phys.}
	\textbf{260}(1), 227--256 (2005).
	
	\bibitem[DZ06]{DZ06}
	C. Dong and Z. Zhao.
	Twisted representations of vertex operator superalgebras.
	\emph{Commun. Contemp. Math.}
	\textbf{8}(1), 101--121 (2006).
	
	\bibitem[DHR71]{DHR71}
	S. Doplicher, R. Haag and J. E. Roberts.
	Local observables and particle statistics. I.
	\emph{Commun. Math. Phys.}
	\textbf{23}(3), 199--230 (1971).
	
	\bibitem[DHR74]{DHR74}
	S. Doplicher, R. Haag and J. E. Roberts.
	Local observables and particle statistics. II.
	\emph{Commun. Math. Phys.}
	\textbf{35}(1), 49--85 (1974).
	
	\bibitem[DGNO10]{DGNO10}
	V. Drinfeld, S. Gelaki, D. Nikshych and V. Ostrik.
	On braided fusion categories. I.
	\emph{Selecta Math. (N.S.)} 
	\textbf{16}(1), 1--119 (2010).
	
	\bibitem[Dun07]{Dun07}
	J. F. R. Duncan.
	Super-moonshine for Conway's largest sporadic group. 
	\emph{Duke Math. J.}
	\textbf{139}(2), 255--315 (2007).
	
	\bibitem[DM15]{DM15}
	J. F. R. Duncan and S. Mack-Crane.
	The moonshine module for Conway's group.
	\emph{Forum Math. Sigma}
	\textbf{3}(e10), pp.\ 52 (2015).
	
	\bibitem[ELMS21]{ELMS21}
	J. van Ekeren, C. H. Lam, S. M\"{o}ller and H. Shimakura.
	Schellekens’ list and the very strange formula.
	\emph{Adv. Math.}
	\textbf{380}, Paper No.\ 107567, pp.\ 33 (2021). 
	
	\bibitem[EMS20]{EMS20}
	J. van Ekeren, S. M\"{o}ller and N. R. Scheithauer.
	Construction and classification of holomorphic vertex operator algebras.
	\emph{J. reine angew. Math.}
	\textbf{759}, 61--99 (2020).
	
	\bibitem[EGNO15]{EGNO15}
	P. Etingof, S. Gelaki, D. Nikshych and V. Ostrik.
	\emph{Tensor Categories}.
	(\emph{Math. Surveys Monogr.} \textbf{205},
	American Mathematical Society, Providence, RI, 2015).
	
	\bibitem[FRS89]{FRS89}
	K. Fredenhagen, K.-H. Rehren and B. Schroer.
	Superselection sectors with braid group statistics and exchange algebras. I. General theory.
	\emph{Commun. Math. Phys.}
	\textbf{125}(2), 201--226 (1989).
	
	\bibitem[FRS92]{FRS92}
	K. Fredenhagen, K.-H. Rehren and B. Schroer.
	Superselection sectors with braid group statistics and exchange algebras. II. Geometric aspects and conformal covariance.
	\emph{Rev. Math. Phys.}
	\textbf{Special Issue}, 113--157 (1992).
	
	\bibitem[FHL93]{FHL93}
	I. B. Frenkel, Y.-Z. Huang and J. Lepowsky.
	On axiomatic approaches to vertex operator algebras and modules.
	\emph{Mem. Amer. Math. Soc.}
	\textbf{104}(494), pp.\ viii+64   (1993).
	
	\bibitem[FLM88]{FLM88}
	I. B. Frenkel, J. Lepowsky and A. Meurman.
	\emph{Vertex Operator Algebras and the Monster}.
	(Academic Press, Inc., London, 1988).
	
	\bibitem[FZ92]{FZ92}
	I. B. Frenkel and Y. Zhu.
	Vertex operator algebras associated to representations of affine and Virasoro algebras.
	\emph{Duke Math. J.}
	\textbf{66}(1), 123--168 (1992).
	
	\bibitem[FRS04]{FRS04}
	J. Fuchs, I. Runkel and C. Schweigert.
	TFT construction of RCFT correlators. III. Simple currents.
	\emph{Nuclear Phys. B} 
	\textbf{694}(3), 277--353 (2004).
	
	\bibitem[GF93]{GF93}
	F. Gabbiani and J. Fr\"{o}hlich.
	Operator algebras and conformal field theory.
	\emph{Commun. Math. Phys.}
	\textbf{155}(3), 569--640 (1993).
	
	\bibitem[GN03]{GN03}
	M. R. Gaberdiel and A. Neitzke.
	Rationality, quasirationality and finite $W$-algebras.
	\emph{Comm. Math. Phys.}
	\textbf{238}(1--2), 305--331 (2003).
	
	\bibitem[GJ22]{GJ22}
	D. Gaiotto and T. Johnson-Freyd.
	Holomorphic SCFTs with small index.
	\emph{Canad. J. Math.} 
	\textbf{74}(2), 573--601 (2022).
	
	\bibitem[Gui19I]{Gui19I}
	B. Gui.
	Unitarity of the modular tensor categories associated to unitary vertex operator algebras, I.
	\emph{Commun. Math. Phys.}
	\textbf{366}(1), 333--396 (2019). 
	
	\bibitem[Gui19II]{Gui19II}
	B. Gui.
	Unitarity of the modular tensor categories associated to unitary vertex operator algebras, II.
	\emph{Commun. Math. Phys.}
	\textbf{372}(3), 893--950 (2019). 
	
	\bibitem[Gui]{Gui}
	B. Gui.
	Unbounded field operators in categorical extensions of conformal nets.
	\href{https://doi.org/10.48550/arXiv.2001.03095}{\textit{arXiv:2001.03095v3} \textbf{[math.QA]}}. 
	(2020).
	
	\bibitem[Gui21]{Gui21}
	B. Gui.
	Categorical extensions of conformal nets.
	\emph{Commun. Math. Phys.}
	\textbf{383}(2), 763--839 (2021). 
	
	\bibitem[Gui(b)]{Guib}
	B. Gui.
	On a Connes fusion approach to finite index extensions of conformal nets.
	\href{https://doi.org/10.48550/arXiv.2112.15396}{\emph{arXiv:2112.15396v1} \textbf{[math.OA]}}.
	 (2021).
	
	\bibitem[Gui22]{Gui22}
	B. Gui.
	Q-systems and extensions of completely unitary vertex operator algebras.
	\emph{Int. Math. Res. Not. IMRN} 
	\textbf{2022}(10), 7550--7614 (2022).
	
	\bibitem[Gui24]{Gui24}
	B. Gui.
	Geometric positivity of the fusion products of unitary vertex operator algebra modules.
	\emph{Comm. Math. Phys.}
	\textbf{405}(3), Paper No.\ 72, pp.\ 65 (2024).
	
	\bibitem[GL96]{GL96}
	D. Guido and R. Longo.
	The conformal spin and statistics theorem.
	\emph{Comm. Math. Phys.} 
	\textbf{181}(1), 11--35 (1996).
	
	\bibitem[Haa96]{Haa96}
	R. Haag.
	\emph{Local Quantum Physics. Fields, Particles, Algebras. Second Revised and Enlarged Edition}.
	(Springer-Verlag Berlin Heidelberg New York, 1996).
	
	\bibitem[HA15]{HA15}
	J. Han and C. Ai.
	Three equivalent rationalities of vertex operator superalgebras
	\emph{J. Math. Phys.}
	\textbf{56}, 111701 (2015).
	
	\bibitem[H{\"o}h95]{Hoe95}
	G. H{\"o}hn.
	\textit{Selbstduale Vertexoperatorsuperalgebren und das Babymonster. (Self-dual Vertex Operator Super Algebras and the Baby Monster.)}
	(Ph.D. thesis, Bonn University, Germany, 1995).
	\href{https://doi.org/10.48550/arXiv.0706.0236}{\textit{arXiv:0706.0236v1} \textbf{[math.QA]}}. (2007).
	
	\bibitem[H{\"o}h10]{Hoe10}
	G. H{\"o}hn.
	The group of symmetries of the shorter Moonshine module.
	\emph{Abh. Math. Semin. Univ. Hambg.} 
	\textbf{80}(2), 275--283 (2010).
	
	\bibitem[H\"{o}h]{Hoe}
	G. H\"{o}hn.
	On the genus of the Moonshine module.
	\href{https://doi.org/10.48550/arXiv.1708.05990}{\emph{arXiv:1708.05990v1} \textbf{[math.QA]}}.
	(2017).
	
	\bibitem[HM22]{HM22}
	G. H{\"o}hn and S. M{\"o}ller.
	Systematic orbifold constructions of Schellekens' vertex operator algebras from Niemeier lattices.
	\emph{J. Lond. Math. Soc. (2)}
	\textbf{106}(4), 3162--3207 (2022).
	
	\bibitem[HM]{HM}
	G. H{\"o}hn and S. M{\"o}ller.
	Classification of self-dual vertex operator superalgebras of central charge at most 24.
	\href{https://arxiv.org/abs/2303.17190}{\textit{arXiv:2303.17190v2} \textbf{[math.QA]}}. (2024).	
	
	\bibitem[Hua95]{Hua95}
	Y.-Z. Huang. 
	A theory of tensor products for module categories for a vertex operator algebra. IV. 
	\emph{J. Pure Appl. Algebra}
	\textbf{100}(1--3), 173--216 (1995). 
	
	\bibitem[Hua96]{Hua96}
	Y.-Z. Huang.
	A nonmeromorphic extension of the Moonshine module vertex operator algebra.
	\emph{In:} Moonshine, the Monster, and Related Topics 
	(\emph{Contemp. Math.} \textbf{193}, American Mathematical Society, Providence, RI, 1996), 123--148.
	
	\bibitem[Hua08]{Hua08}
	Y.-Z. Huang. 
	Rigidity and modularity of vertex tensor categories.
	\emph{Commun. Contemp. Math.} 
	\text{10}(1), 871--911 (2008).
	
	\bibitem[HKL15]{HKL15}
	Y.-Z. Huang, A. Kirillov and J. Lepowsky.
	Braided tensor categories and extensions of vertex operator algebras.
	\emph{Commun. Math. Phys.}
	\textbf{337}(3), 1143--1159 (2015).
	
	\bibitem[HL95I]{HL95I}
	Y.-Z. Huang, J. Lepowsky.
	A theory of tensor products for module categories for a vertex operator algebra, I.
	\emph{Selecta Math. (N.S.)}
	\textbf{1}(4), 699--756 (1995). 
	
	\bibitem[HL95II]{HL95II}
	Y.-Z. Huang, J. Lepowsky.
	A theory of tensor products for module categories for a vertex operator algebra, II.
	\emph{Selecta Math. (N.S.)}
	\textbf{1}(4), 757--786 (1995). 
	
	\bibitem[HL95III]{HL95III}
	Y.-Z. Huang and J. Lepowsky.
	A theory of tensor products for module categories for a vertex operator algebra, III.
	\emph{J. Pure Appl. Algebra}
	\textbf{100}(1-3), 141--171 (1995).
	
	\bibitem[JS93]{JS93}
	A. Joyal and R. Street.
	Braided tensor categories.
	\emph{Adv. Math.}
	\textbf{102}(1), 20--78 (1993).
	
	\bibitem[Kac01]{Kac01}
	V. G. Kac.
	\emph{Vertex Algebras for Beginners}.
	(ULS \textbf{10}, American Mathematical Society, Providence, RI, 2001).
	
	\bibitem[KLM01]{KLM01}
	Y. Kawahigashi, R. Longo and M. M{\"u}ger.
	Multi-interval subfactors and modularity of representations in conformal field theory.
	\emph{Comm. Math. Phys.}
	\textbf{219}(3), 631--669 (2001).
	
	\bibitem[KO02]{KO02}
	A. Jr. Kirillov and V. Ostrik.
	On a $q$-analogue of the McKay correspondence and the ADE classification of $\widetilde{\mathfrak{sl}}_2$ conformal field theories.
	\emph{Adv. Math.}
	\textbf{171}(2), 183--227 (2002).
	
	\bibitem[KM15]{KM15}
	M. Krauel and M. Miyamoto.
	A modular invariance property of multivariable trace functions for regular vertex operator algebras.
	\emph{J. Algebra}
	\textbf{444}, 124--142 (2015).
	
	\bibitem[Lam20]{Lam20}
	C. H. Lam.
	Cyclic orbifolds of lattice vertex operator algebras having group-like fusions.
	\emph{Lett. Math. Phys.}\textbf{110}(5), no.5, 1081--1112 (2020).
	
	\bibitem[Lam23]{Lam23}
	C. H. Lam.
	Unitary forms for holomorphic vertex operator algebras of central charge 24.
	\emph{Lett. Math. Phys.}
	\textbf{113}(2), Paper No.\ 28, pp.\ 28 (2023).
	
	\bibitem[LM23]{LM23}
	C. H. Lam and M. Miyamoto.
	A lattice theoretical interpretation of generalized deep holes of the Leech lattice vertex operator algebra.
	\emph{Forum Math. Sigma}
	\textbf{11}, Paper No.\ e86, pp.\ 36 (2023).
	
	\bibitem[LS19]{LS19}
	C. H. Lam and H. Shimakura.
	71 holomorphic vertex operator algebras of central charge 24.
	\emph{Bull. Inst. Math. Acad. Sin. (N. S.)}
	\textbf{14}(1), 87--118 (2019).
	
	\bibitem[LS20]{LS20}
	C. H. Lam and H. Shimakura.
	Inertia groups and uniqueness of holomorphic vertex operator algebras.
	\emph{Transform. Groups}
	\textbf{25}(4), 1223--1268 (2020).
	
	\bibitem[LL04]{LL04}
	J. Lepowsky and H.-S. Li.
	\emph{Introduction to Vertex Operator Algebras and Their Representations}.
	(Birkhäuser Boston, Inc., Boston, 2004).
	
	\bibitem[Li94]{Li94}
	H. Li.
	\emph{Representation Theory and Tensor Product Theory for Vertex Operator Algebras}.
	(Ph.D. Thesis, New Brunswick, USA, 1994).
	\href{https://doi.org/10.48550/arXiv.hep-th/9406211}{\emph{arXiv:hep-th/9406211v1}}.
	
	\bibitem[Li99]{Li99}
	H. Li.
	Some finiteness properties of regular vertex operator algebras.
	\emph{J. Algebra} 
	\textbf{212}(2), 495--514 (1999).
	
	\bibitem[Lon03]{Lon03}
	R. Longo.
	Conformal subnets and intermediate subfactors.
	\emph{Commun. Math. Phys.}
	\textbf{237}(1-2), 7--30 (2003).
	
	\bibitem[LX04]{LX04}
	R. Longo and F. Xu.
	Topological sectors and a dichotomy in conformal field theory.
	\emph{Comm. Math. Phys.} 
	\textbf{251}(2), 321--364 (2004).
	
	\bibitem[Mas14]{Mas14}
	G. Mason.
	Lattice subalgebras of strongly regular vertex operator algebras.
	\emph{In:} Conformal Field Theory, Automorphic Forms and Related Topics, 
	(\emph{Contrib. Math. Comput. Sci.} \textbf{8},	Springer, Heidelberg, 2014),
	31--53. 
	
	\bibitem[McR20]{McR20}
	R. McRae.
	On the tensor structure of modules for compact orbifold vertex operator algebras.
	\emph{Math. Z.}
	\textbf{296}(1--2), 409--452 (2020).
	
	\bibitem[Miy04]{Miy04}
	M. Miyamoto.
	A new construction of the Moonshine vertex operator algebra over the real number field.
	\emph{Ann. of Math.} 
	\textbf{159}(2), 535–596 (2004).
	
	\bibitem[Miy15]{Miy15}
	M. Miyamoto.
	$C_2$-cofiniteness of cyclic-orbifold models.
	\emph{Comm. Math. Phys.}
	\textbf{335}(3), 1279--1286 (2015).
	
	\bibitem[MS23]{MS23}
	S. M{\"o}ller and N. R. Scheithauer.
	Dimension formulae and generalised deep holes of the Leech lattice vertex operator algebra.
	\emph{Ann. of Math. (2)}
	\textbf{197}(1), 221--288 (2023).
	
	\bibitem[MS24]{MS24}
	S. M\"{o}ller and N. R. Scheithauer.
	A geometric classification of the holomorphic vertex operator algebras of central charge 24.
	\emph{Algebra Number Theory}, to appear (2024).
	\href{https://doi.org/10.48550/arXiv.2112.12291}{\emph{arXiv:2112.12291v1} \textbf{[math.QA]}}.	
	
	\bibitem[MS]{MS}
	G. W. Moore and R. K. Singh.
	Beauty and the beast part 2: apprehending the missing supercurrent.
	\href{https://doi.org/10.48550/arXiv.2309.02382}{\textit{arXiv:2309.02382v1} \textbf{[hep-th]}}. (2023).
	
	\bibitem[MTW18]{MTW18}
	V. Morinelli, Y. Tanimoto and M. Weiner.
	Conformal covariance and the split property.
	\emph{Comm. Math. Phys.}
	\textbf{357}(1), 379--406 (2018).
	
	\bibitem[Nel59]{Nel59}
	E. Nelson.
	Analytic vectors.
	\emph{Ann. of Math.}
	\textbf{70}(3), 572--615 (1959).
	
	\bibitem[Nik80]{Nik80}
	V. V. Nikulin.
	Integral symmetric bilinear forms and some of their applications.
	\emph{Math. USSR Izv.} 
	\textbf{14}(1), 103--167 (1980).
	
	\bibitem[Ped89]{Ped89}
	G. K. Pedersen.
	\emph{Analysis Now}.
	(Springer-Verlag, New York, 1989).
	
	\bibitem[Ray24]{Ray24}
	B. C. Rayhaun.
	Bosonic rational conformal field theories in small genera, chiral fermionization, and symmetry/subalgebra duality.
	\emph{J. Math. Phys.}
	\textbf{65}(5), Paper No.\ 052301, pp.\ 129 (2024).
	
	\bibitem[RTT22]{RTT22}
	C. Raymond, Y. Tanimoto and J. E. Tener. 
	Unitary vertex algebras and Wightman conformal field theories.
	\emph{Commun. Math. Phys.} 
	\textbf{395}, 299-330 (2022).
	
	\bibitem[RS80]{RS80}
	M. Reed and B. Simon.
	\emph{Methods of Modern Mathematical Physics II: Functional Analysis. Revised and Enlarged Edition.}
	(Academic Press, London, UK, 1980).
	
	\bibitem[Sch93]{Sch93}
	A. N. Schellekens.
	Meromorphic $c=24$ conformal field theories.
	\emph{Comm. Math. Phys.}
	\textbf{153}(1), 159--185 (1993).
	
	\bibitem[SW64]{SW64}
	R. F. Streater and A. S. Wightman.
	\emph{PCT, Spin and Statistics, and All That}.
	(W. A. Benjamin, Inc., New York, 1964).
	
	\bibitem[Ten18]{Ten18}
	J. E. Tener.
	Positivity and fusion of unitary modules for unitary vertex operator algebras.
	\emph{RIMS Kôkyûroku}, no. 2086, 6--13 (2018).
	
	\bibitem[Ten19]{Ten19}
	J. E. Tener.
	Geometric realization of algebraic conformal field theories.
	\emph{Adv. Math.}
	\textbf{349}, 488--563 (2019).
	
	\bibitem[Ten19b]{Ten19b}
	J. E. Tener.
	Representation theory in chiral conformal field theory: from fields to observables.
	\emph{Selecta Math. (N. S.)}
	\textbf{25}(5)(76), pp.\  82 (2019).
	
	\bibitem[Tur10]{Tur10}
	V. G. Turaev.
	\emph{Quantum Invariants of Knots and 3-manifolds. Third Edition}.
	(\emph{De Gruyter Stud. Math.} \textbf{18},	De Gruyter, Berlin, 2016).
	
	\bibitem[Was95]{Was95}
	A. Wassermann. 
	Operator algebras and conformal field theory. 
	\emph{In:} 
	Proceedings of the International Congress of Mathematicians, Vol. 1,2 (Z{\"u}rich, 1994).
	(Birkh{\"a}user Verlag, Basel, 1995), 966--979.
	
	\bibitem[Was98]{Was98}
	A. Wassermann.
	Operator algebras and conformal field theory. III. Fusion of positive energy representations of $\operatorname{LSU}(N)$ using bounded operators.
	\emph{Invent. Math.} 
	\textbf{133}(3), 467--538 (1998).
	
	\bibitem[Xu98]{Xu98}
	X. Xu.
	\emph{Introduction to Vertex Operator Superalgebras and Their Modules.}
	(\emph{Math. Appl.} \textbf{456}, Kluwer Academic Publishers, Dordrecht, 1998).
	
	\bibitem[Yam05]{Yam05}
	H. Yamauchi.
	2A-orbifold construction and the baby-monster vertex operator superalgebra.
	\emph{J. Algebra} 
	\textbf{284}(2), 645--668 (2005).
	
	\bibitem[Zhu96]{Zhu96}
	Y. Zhu.
	Modular invariance of characters of vertex operator algebras.
	\emph{J. Amer. Math. Soc.}
	\textbf{9}(1), 237--302 (1996).
	
\end{thebibliography}
\end{document}